\documentclass[11pt, a4paper]{article}
\usepackage{amsmath}
\usepackage{amsthm}%% The amsthm package provides extended theorem environments
\usepackage{amssymb}
\usepackage{array}
\usepackage{multirow}
\usepackage{graphicx}
\usepackage{epsfig}
\usepackage{mathrsfs}
\usepackage{hhline}
\usepackage{longtable}
\usepackage{array}
\usepackage{enumerate}
\usepackage{xcolor}
\usepackage{float}

\usepackage{graphicx}
\graphicspath{%
    {converted_graphics/}% inserted by PCTeX
    {/}% inserted by PCTeX
}

\newcolumntype{"}{@{\hskip\tabcolsep\vrule width 1pt\hskip\tabcolsep}}

\usepackage[mathlines]{lineno}
\modulolinenumbers[2]
\newcommand*\patchAmsMathEnvironmentForLineno[1]{%
\expandafter\let\csname old#1\expandafter\endcsname\csname #1\endcsname  \expandafter\let\csname oldend#1\expandafter\endcsname\csname end#1\endcsname  \renewenvironment{#1}%
{\linenomath\csname old#1\endcsname}%
{\csname oldend#1\endcsname\endlinenomath}}%
\newcommand*\patchBothAmsMathEnvironmentsForLineno[1]{%
\patchAmsMathEnvironmentForLineno{#1}%
\patchAmsMathEnvironmentForLineno{#1*}}%
\AtBeginDocument{%
\patchBothAmsMathEnvironmentsForLineno{equation}%
\patchBothAmsMathEnvironmentsForLineno{align}%
\patchBothAmsMathEnvironmentsForLineno{flalign}%
\patchBothAmsMathEnvironmentsForLineno{alignat}%
\patchBothAmsMathEnvironmentsForLineno{gather}%
\patchBothAmsMathEnvironmentsForLineno{multline}%
}

\def\Z{\mathbb{Z}}
\def\N{\mathbb{N}}

\definecolor{vividviolet}{rgb}{0.62, 0.0, 1.0}
\def\nt{\noindent}
\def\ms{\medskip}
\def\rsq{\hspace*{\fill}$\blacksquare$\medskip}
\def\di{\displaystyle}
\newtheoremstyle{de}%name
  {10pt}          % space above
  {10pt}  % space below
  {\rm}  % bofy font
  {}%{\parindent}     % ident - empty=no indent,  \parindent= paragraph indent
  {\bf}  % thm head font
  {. }    % punctuation after thm head
  { }    % space after thm head: `` ``=normal \newline=linebreak
  {}     % thm head specification
\theoremstyle{de}

\newtheorem{de}{Definition}[section]
\newtheorem{example}{Example}[section]
\newtheorem{rem}[de]{Remark}
\newtheorem{problem}{Problem}[section]

\newtheoremstyle{theorem}%name
  {10pt}          % space above
  {10pt}  % space below
  {\it}  % bofy font
  {}%{\parindent}     % ident - empty=no indent,  \parindent= paragraph indent
  {\bf}  % thm head font
  {. }    % punctuation after thm head
  { }    % space after thm head: `` ``=normal \newline=linebreak
  {}     % thm head specification
\theoremstyle{theorem}

\numberwithin{equation}{section}
\def\Z{\mathbb{Z}}
\def\N{\mathbb{N}}

\newtheorem{theorem}{Theorem}[section]
\newtheorem{lemma}[theorem]{Lemma}%[section]

\numberwithin{equation}{section}

\def\Lau{\color{red}}

\begin{document}
\baselineskip18truept
\normalsize
\begin{center}
{\mathversion{bold}\Large \bf Necessary and sufficient conditions of a class of bipartite graphs \\with local antimagic chromatic number 2 - an algebraic approach}

\bigskip
{\large  G.C. Lau{$^{a,}$}\footnote{Corresponding author}, W.C. Shiu{$^{b}$} }\\

\medskip

\emph{{$^a$}77D, Jalan Suboh,}\\
%\emph{Universiti Teknologi MARA (Segamat Campus),}\\
\emph{85000, Johor, Malaysia.}\\
\emph{geeclau@yahoo.com}\\

\medskip

\emph{{$^b$}Department of Mathematics, The Chinese University of Hong Kong,}\\
\emph{Shatin, Hong Kong.}\\
\emph{wcshiu@associate.hkbu.edu.hk}\\

\end{center}

%\quad {\Lau Lau}\quad {\Shiu Shiu}\quad {\magenta Shiu or Question} \quad {\cyan Maybe no use}

\medskip
\begin{abstract}For a connected graph $G = (V, E)$,  a bijective edge labeling $f:E \to\{1,\ldots ,|E|\}$ is a local antimagic labeling of $G$ if it induces a vertex labeling $f^+$ such that for any pair of adjacent vertices $x$ and $y$, $f^+(x)\not= f^+(y)$, where the induced vertex label $f^+(x)= \sum f(xu)$, with $u$ ranging over all the vertices adjacent to $x$.  The minimum number of distinct induced vertex labels over all local antimagic labelings of $G$ is the local antimagic chromatic number of $G$, denoted $\chi_{la}(G)$.  In this paper, we make use of algebraic analysis to obtain necessary and sufficient conditions for every bipartite graph with all vertices of degree 2 except exactly three vertices of degree at least $3$ (say, $u,v,w$, such that  every $(u,v)$-, $(u,w)$- and $(v,w)$-path is  of even length and not all are of length 2) to have local antimagic chromatic number 2. Moreover, we showed that the consecutive edge labels of every induced path of each case is unique.
%For a connected graph $G = (V, E)$,  a bijective edge labeling $f:E \to\{1,\ldots ,|E|\}$ is a local antimagic labeling of $G$ if it induces a vertex labeling $f^+$ such that for any pair of adjacent vertices $x$ and $y$, $f^+(x)\not= f^+(y)$, where the induced vertex label $f^+(x)= \sum f(xu)$, with $u$ ranging over all the vertices adjacent to $x$.  The minimum number of distinct induced vertex labels over all local antimagic labelings of $G$ is the local antimagic chromatic number of $G$, denoted $\chi_{la}(G)$. In this paper, we make use of algebraic analysis to obtain necessary and sufficient conditions for every bipartite graph with all vertices of degree 2 except exactly three vertices of degree at least 3 to have local antimagic chromatic number 2. Moreover, we showed that the consecutive edge labels of every induced path of each case is unique.

\medskip

\end{abstract}
%\normalsize

\tolerance=10000
\baselineskip12truept
%\newbox\thebox
%\global\setbox\thebox=\vbox to 0.2truecm{\hsize
%0.15truecm\noindent\hfill}
%\def\boxit#1{\vbox{\hrule\hbox{\vrule\kern0pt
%     \vbox{\kern0pt#1\kern0pt}\kern0pt\vrule}\hrule}}
%\def\qed{\lower0.1cm\hbox{\noindent \boxit{\copy\thebox}}}
\def\qed{\hspace*{\fill}$\Box$\medskip}

\def\s{\,\,\,}
\def\ss{\smallskip}
\def\ms{\medskip}
\def\bs{\bigskip}
\def\c{\centerline}
\def\nt{\noindent}
\def\ul{\underline}
\def\lc{\lceil}
\def\rc{\rceil}
\def\lf{\lfloor}
\def\rf{\rfloor}
\def\a{\alpha}
\def\b{\beta}
\def\n{\nu}
\def\o{\omega}
\def\ov{\over}
\def\m{\mu}
\def\t{\tau}
\def\th{\theta}
\def\k{\kappa}
\def\l{\lambda}
\def\L{\Lambda}
\def\g{\gamma}
\def\d{\delta}
\def\D{\Delta}
\def\e{\epsilon}
\def\lg{\langle}
\def\rg{\tongle}
\def\p{\prime}
\def\sg{\sigma}
\def\to{\rightarrow}

\newcommand{\K}{K\lower0.2cm\hbox{4}\ }
\newcommand{\cl}{\centerline}
\newcommand{\om}{\omega}
\newcommand{\ben}{\begin{enumerate}}

\newcommand{\een}{\end{enumerate}}
\newcommand{\bit}{\begin{itemize}}
\newcommand{\eit}{\end{itemize}}
\newcommand{\bea}{\begin{eqnarray*}}
\newcommand{\eea}{\end{eqnarray*}}
\newcommand{\bear}{\begin{eqnarray}}
\newcommand{\eear}{\end{eqnarray}}

\section{Introduction}

%\nt In 1994, Hartsfield and Ringer~\cite{H+R} introduced the concept of antimagic labeling of a graph $G(V,E)$. A bijective edge labeling $f: E \to \{1, \ldots, |E|\}$ is called an antimagic labeling of $G$ if for any two distinct vertices $u$ and $v$, $w(u)\ne w(v)$, where $w(u) = \sum f(e)$ with $e$ ranging over all the edges incident to $u$. The most famous unsolved problems are the following conjectures~\cite{H+R}.

%\begin{conjecture}  Every connected graph other than $K_2$ is antimagic. \end{conjecture}

%\begin{conjecture} Every tree other than $K_2$ is antimagic. \end{conjecture}

In 2017, Arumugam et al.~\cite{Arumugam} introduced the concept of local antimagic labeling and local antimagic chromatic number. For a connected graph $G$, we say $G$ is {\it local antimagic} if it admits a {\it local antimagic labeling}, i.e., a bijective edge labeling $f : E \to \{1,\ldots ,|E|\}$ such that the induced vertex labeling $f^+ : V \to \Z$ given by $f^+(u) = \sum_{e\in E_u} f(e)$ (where $E_u$ is the set of all the edges incident to $u$) has the property that any two adjacent vertices have distinct induced vertex labels. Thus, $f^+$ is a proper coloring of $G$. By definition, $G$ has order at least 3 The number of distinct induced vertex labels under $f$ is denoted by $c(f)$, and is called the {\it color number} of $f$. The {\it local antimagic chromatic number} of $G$, denoted by $\chi_{la}(G)$, is $\min\{c(f) \;|\; f\mbox{ is a local antimagic labeling of } G\}$. Clearly, $2\le \chi_{la}(G)\le |V(G)|$.  In~\cite{Haslegrave}, Haslegrave proved that the local antimagic chromatic number is well-defined for every connected graph except $K_2$. Thus, every graph without any $K_2$ component admits a local antimagic labeling if we also define $f^+(u) = 0$ for any isolated vertex $u$.

\ms\nt The local antimagic chromatic number of some families of standard graphs are first presented in~\cite{Arumugam}. In~\cite[Corollary 2]{LNS-GC}, Lau et al. completely settled the local antimagic chromatic number of wheels and complete bipartite graphs. In~\cite[Theorem 2.4]{LSN-IJMSI}, the authors completely determined the local antimagic chromatic number of the one-point union of cycles (that has all vertices are of degree 2 except one vertex is of degree at least 3). As a natural extension, the authors in~\cite{Lau+Shiu+Nal+Zhang+Prem} completely characterized bridge graphs (that has all vertices are of degree 2 except two vertices, say $u$ and $v$, are of equal degree at least 3) with local antimagic chromatic number 2.

%\ms\nt A graph consisting of $s$ paths joining two vertices is called an {\it $s$-bridge graph}, which is denoted by $\th(a_1,\dots, a_s)$, where $s\ge 2$ and $1\le a_1\le a_2 \le \cdots \le a_s$ are the lengths of the $s$ paths.  For convenience, we shall let $\th_s = \th(a_1,a_2,\ldots,a_s)$ if there is no confusion. In this paper, we shall characterize $\th_s$ with $\chi_{la}(\theta_s)=2$.

\ms\nt The following lemma in \cite[Lemma 2.1]{LSN-DMGT} or~\cite[Lemma 2.3]{LSN-IJMSI} gives a necessary condition for a bipartite graph $G$ to have $\chi_{la}(G) = 2$.

\begin{lemma}[{\cite[Lemma 2.3]{LSN-IJMSI}}]\label{lem-2part} Let $G$ be a graph of size $q$. Suppose there is a local antimagic labeling of $G$ inducing a $2$-coloring of $G$ with colors $x$ and $y$, where $x<y$. If $X$ and $Y$ are the sets of vertices colored $x$ and $y$, respectively, then $G$ is a bipartite graph with bipartition $(X,Y)$ and $|X|>|Y|$. Moreover,
$x|X|=y|Y|= \frac{q(q+1)}{2}$.
\end{lemma}

%\nt Clearly, $2\le \chi(\th(a_1,a_2,\ldots,a_s)) \le 3$ and the lower bound holds if and only if  $a_1\equiv \cdots \equiv a_s  \pmod{2}$. By Lemma~\ref{lem-2part}, we immediately have the following lemma.

%\begin{lemma}\label{lem-th=2} For $s\ge 2$ and $1\le i\le s$, if $\chi_{la}((\th(a_1,a_2,\ldots,a_s)) =2$, then $a_i \equiv 0 \pmod{2}$. Otherwise, $\chi_{la}((\th(a_1,a_2,\ldots,a_s)) \ge 3$. \end{lemma}

\nt Throughout this paper, we shall use $a^{[n]}$ to denote a sequence of length $n$ in which all terms are $a$, where $n\ge 2$. For integers $1\le a < b$, we let $[a,b]$ denote the set of integers from $a$ to $b$. If $u$ and $v$ are two vertices of degree at least 3 in $G$, then an {\it induced $(u,v)$-path} (or just induced path) of $G$ is a path with end-vertices $u$ or $v$ and all other vertices are of degree 2 in $G$. If $u=v$, we may call it an {\it induced cycle}. For convenience, we shall also call induced cycle as induced path in this paper, whenever possible. % Interested readers may refer to~\cite{Arumugam+L+P+W, LauShiuNg-pendants, LauShiuSoo} for more results related to local antimagic chromatic number of graphs.

\ms\nt Let $\mathcal B_d$ be a class of bipartite graphs such that (i) every vertex of $G \in \mathcal B_d$ is of degree 2 except exactly $d$ vertices, say $v_1, v_2, \ldots, v_d$, $d\ge 1$, are of degree at least $3$, (ii) $\chi_{la}(G) = 2$, and (iii) every induced $(v_i,v_j)$-path $(1\le i,j\le d)$ (or cycle) of $G$ has even number of edges. Therefore, all the one-point union of cycles in~\cite[Theorem 2.4]{LSN-IJMSI} with local antimagic chromatic number 2 are in $\mathcal B_1$, while all the bridge graphs in~\cite[Theorem 2.1]{Lau+Shiu+Nal+Zhang+Prem}  are in  $\mathcal B_2$.

\ms\nt Recall that from the proof of~\cite[Theorem 2.1]{Lau+Shiu+Nal+Zhang+Prem}, every non $K_{2,s}$ graphs is constructed from $s\ge 4$ distinct paths by merging $s$ suitable end-vertices, no two of which belong to the same paths, to form two vertices of degree $s$. We note that if we can partition the labels of the end-edges into two sets with sum of elements equal $y$ such that at least a set contains at least two elements belong to the same path of length at least 4, and merge the corresponding end-vertices into two distinct vertices of degree $s$, then the graph such obtained will be a non-bridge graph in $\mathcal B_2$ with local antimagic chromatic number 2. Consequently, if the proof is not restricted to obtaining bridge graphs only, then we can get all the possible paths and a labeling that allow us to obtain all the graphs in $\mathcal B_2$. If we can find a systematic algorithm of getting all the required bipartitions of the end-edge labels, we then obtain all the graphs in $\mathcal B_2$. This will be discussed in another paper.
\ms\nt For $m\ge 2$, $i\ge 1$ and $n_i\ge 1$, let $B(n_1,n_2,\ldots,n_m)$ be the union of $K_{2,n_i}$ with bipartition $(X_i, Y_i)$, where $X_i=\{x_{i-1},x_i\}$, $Y_i=\{y_{i,1},y_{i,2},\dots, y_{i,n_i}\}$ and $x_m=x_0$. Note that the graph $B(n^{[3]})$ in~\cite[Theorem 3.23]{LSN-DMGT} is a graph in $\mathcal B_3$ with $\chi_{la}(B(n^{[3]})) = 2$.

\ms\nt In this paper, we extend the study to the determination of all the graphs in $\mathcal B_3$. In what follows, we only consider $G\not\cong B(n_1, n_2, n_3)$ for some $n_1, n_2, n_3\ge 1$.  This is an attempt to provide partial solutions to~\cite[Problem 4.1]{Lau+Shiu+Nal+Zhang+Prem}. We shall in another paper investigate the local antimagic chromatic number of $B(n_1,n_2,\ldots, n_m)$  \cite[Problem 3.26]{LSN-DMGT}.
%\nt In this section, we assume $\chi_{la}(\th_s)=2$. So by Lemma~\ref{lem-th=2}, $\th_s=\th(a_1, \dots, a_s)$ is bipartite and all $a_i$ are even. When $s=2$, $\th_s$ is a cycle, whose local antimagic chromatic number is 3. Thus $s\ge 3$.

\ms\nt For integers $i$ and $d$ and positive integer $s$, let $A_s(i;d)$ be the arithmetic progression of length $s$ with common difference $d$ and first term $i$. We first have a useful lemma.

\begin{lemma}[{\cite[Lemma~2.1]{Lau+Shiu+Nal+Zhang+Prem}}]\label{lem-AP}
Suppose $s, d\in\N$.
\begin{enumerate}[(a)]
\item For $i,j\in\Z$, the sum of the $k$-th term of $A_s(i;d)$ and that of $A_s(j;-d)$ is $i+j$ for $k\in[1,s]$; and the sum of the $k$-th term of $A_s(i;d)$ and the $(k-1)$-st term of $A_s(j;-d)$ is $i+j+d$ for $k\in[2,s]$.
\item If $0<|i_1-i_2|<d$, then $A_s(i_1; d)\cap A_s(i_2,\pm d)=\varnothing$.
\end{enumerate}
\end{lemma}

\nt Suppose $A_1$ and $A_2$ are two sequences of length $n$. We combine these two sequences as a sequence of length $2n$, denoted $A_1\diamond A_2$, whose $(2i-1)$-st term is the $i$-th term of $A_1$ and the $(2i)$-th term is the $i$-th term of $A_2$, $1\le i\le n$.

%\newpage

\section{Main Result}

\ms\nt We extend the ideas used in proving Theorem 2.3 in~\cite{Lau+Shiu+Nal+Zhang+Prem}. For completeness, we include the details as much as possible. We shall keep the notation defined in Section 1. We refer to \cite{Bondy} for concepts and notation not defined in this paper.

\ms\nt For $d\ge 1$, let $\widetilde{\mathcal B}_d=\mathcal B_d\setminus \{B(n_1,\dots, n_d)\;|\; n_i\ge 1, 1\le i\le d\}$. Let $f$ be an edge labeling of $G\in \widetilde{\mathcal B}_d$ with $s\ge 2$ distinct induced paths (or cycles). We shall call the $2s$ edges incident to $v_i$, $1\le i\le d,$ as {\it end-edges} and the numbers assigned to the end-edges are called {\it end-edge labels}.

\begin{lemma}\label{lem-induced} Suppose $G\in\widetilde{\mathcal B}_d$, $d\ge 1$, has size $m$. If $f$ is a local antimagic $2$-coloring of $G$ with induced vertex labels $x,y$, then
\begin{enumerate}[(a)]
\item $f^+(v_i)= y > x$ for $1\le i\le d$,
\item $m\ge 3d+2$;
\item $x=m+1$ is odd and $y=m(m+1)/(m-2s+2d) \ge (2s^2+s)/d$;
\item all integers in $[1, y-x]$ are end-edge labels and $y/2$ is also an end-edge label if $y$ is even;
\item there exist $b\ge a\ge 2$ such that $a\in \{2i\;|\; 1\le i\le 2d\}$ with $a = y+4d-4s-1-t$, $b = y+4d-4s-1+t$, $a+b=2y+8d-8s-2$ and $ab = 8(2s^2-4sd+2d^2+s-d)$, where $t\ge 0$.
\item $m=(b+4s-4d)/2$.
\end{enumerate}
 \end{lemma}

\begin{proof} Note that $m$ must be even.  Thus, the sum of degrees of the vertices $v_i$, $1\le i\le d$, of degree at least 3 is $2s\ge 3d$ and $m\ge 2s+2\ge 3d+2$.  Moreover, $G$ has order $m-s+d$. Let $X$ and $Y$ be the set of vertices of $G$ with induced labels $x$ and $y$, respectively. Thus $(X,Y)$ is a bipartition of $G$. Without loss of generality, we may assume $f^+(v_i) = y$ for  $1\le i\in d$. Thus, $|X| = m/2$ and $|Y| = m/2 -s + d$.
By Lemma~\ref{lem-2part}, we have $|X|\ne |Y|$, $y>x$ and $x|X| = y|Y| = m(m+1)/2$. Hence, $x = m+1$ is odd, and $y=m(m+1)/(m-2s+2d) \ge (1+2+\cdots + 2s)/d = (2s^2+s)/d$.

\ms\nt Suppose there is a non-end-edge $z_1z_2$ with $f(z_1z_2)=l$. Without loss of generality, we may assume $f^+(z_1)=x$ so that $f^+(z_2)=y$. Since $z_1z_2$ is not an end-edge, $z_2$ is of degree 2 and there is another vertex, say $z_3$, is adjacent to $z_2$. Thus, $f(z_2z_3)=y-l$. Since $1\le y - l \le m$, we have $l \ge y - m = y - x+1$. Thus, all integers in $[1, y-x]$ must be assigned to end-edges. Therefore, $y-x \le 2s$. Moreover, when $y$ is even, we have $y/2$ must be an end-edge label.

\ms\nt Solving $m(m+1) = y(m - 2s+2d)$ for $m$, we get $m^2 - (y-1)m + y(2s-2d) = 0$ so that $m = \frac{1}{2}[(y - 1) \pm \sqrt{y^2+(8d-8s-2)y+1}]$. Thus, $y^2 + (8d-8s-2)y+1 = t^2 \ge 0$, where $t$ is a nonnegative integer. This gives $(y+4d-4s-1)^2 - t^2 = (4d-4s-1)^2 -1$ or $(y+4d-4s-1-t)(y+4d-4s-1+t) = (4s-4d+2)(4s-4d)= 8(s-d)(2s-2d+1)$. Let $a = y+4d-4s-1-t$, $b = y+4d-4s-1+t$, we have $a+b=2y+8d-8s-2$ and $ab = 8(2s^2-4sd+2d^2+s-d)$ with $b\ge a > 0$. Clearly, both $a,b$ are even. Moreover, $y=4s-4d+1+\frac{a+b}{2}$.

\ms\nt Recall that $y - (2s^2+s)/d\ge 0$. Now,
\begin{align*}y -(2s^2+s)/d & = 4s-4d+1+\frac{a+b}{2}-\frac{2s^2+s}{d}\nonumber\\
& = \frac{a+b}{2} - \frac{2s^2-4sd+2d^2+s-d}{d}-2d=\frac{a+b}{2} -\frac{ab}{8d}-2d\nonumber\\
& = \frac{4ad+4bd-ab-16d^2}{8d} = -\frac{(a-4d)(b-4d)}{8d}. %\label{eq-y0}
\end{align*}
Thus, $a\in \{2i\;|\; 1\le i\le 2d\}$.

\nt Suppose $m=\di \frac{y-1-t}{2}=\frac{4s-4d+a}{2}\le \frac{4s}{2}=2s$. This is impossible since $m\ge 2s+2$. Thus $m=\di \frac{y-1+t}{2} =\frac{b+4s-4d}{2}$.
\end{proof}

\nt In what follows, we consider $d=3$. Thus, $y=4s-11+\frac{a+b}{2}$ and $ab=8(2s^2-11s+15)$. So
\begin{equation}\label{eq-y}\left.\begin{aligned}
b &=\frac{8(2s^2-11s+15)}{a},\\
 y& =4s-11+\frac{a}{2}+\frac{4(2s^2-11s+15)}{a},\\
m& =\frac{4(2s^2-11s+15)}{a}+2s-6.
\end{aligned}\right\}\end{equation}

 \nt For convenience, throughout this paper, we let
\[D'=\begin{cases}
 [1,y-x] & \mbox{ if $y$ is odd,}\\
 [1, y-x]\cup\{y/2\} & \mbox{ if $y$ is even,}\end{cases}\] which is a subset of the actual end-edge label set $D$.

\begin{theorem}\label{thm-iff} A graph $G$ with $s\ge 5$ induced paths (or cycles) is in $ \widetilde{\mathcal B}_3$ if and only if
\begin{table}[H]
\fontsize{9}{12}\selectfont
\begin{tabular}{l|l|l|l|l}
Case & $s$ & Induced paths & Induced paths & Induced path\\\hline
A-1 & $s=6l+3\ge 9$ & $l$ paths of length $4l+2$ & $5l+3$ paths of length $4l$ & not applicable \\
A-2 & $s=6l+1\ge 7$ & $3l$ paths of length $4l$ & $3l+1$ paths of length $4l-2$& not applicable \\\hline
B-1 & $s=5l\ge 10$ & $5l-1$ paths of length $4l-2$ & one path of length $2l-2$ & not applicable \\
B-2 & $s=5l+3\ge 8$ & $l$ paths of length $4l+2$ & $4l+2$ paths of length $4l$ & a path of length $2l$\\\hline
C$^*$ & $s=4l+3\ge 7$ & $l$ path of length $4l+2$ & $3l+2$ paths of length $4l$ & not applicable \\\hline
D-1$^*$ & $s=3l\ge 6$ & $l-1$ paths of length $4l-2$ & $2l-1$ paths of length $4l-4$ & a path of length $2l-2$\\
D-2$^*$ & $s=3l+1\ge 7$ & $3l-1$ paths of length $4l-2$ & one path of length $2l-2$ & not applicable\\\hline
E$^\dag$ & $s=2l+1\ge 5$ & $l-1$ paths of length $4l-2$ & $l$ paths of length $4l-4$ & not applicable \\\hline
F$^\dag$ & & $s-3$ path of length $4s-10$ & one path of length $2s-6$ & not applicable \\\hline
\end{tabular}\\
\normalsize
\nt $^*$ One of these paths is split into two paths.\\
\nt $^\dag$ Two of these paths are split into two paths respectively, or else one of these paths is split into three paths.
\caption{Exact value of $s$ and the corresponding number of paths and lengths.}\label{table-1}
\end{table}

\nt Moreover, the edge labels are unique (up to isomorphism) as follows:
\begin{table}[H]
\fontsize{10}{14}\selectfont
\begin{tabular}{l|l|l}
Case & Path length & Consecutive edge labels      \\\hline
A-1 & $4l+2$ & $A_{2l+1}(i;12l+6)\diamond A_{2l+1}(24l^2+14l+1-i; -12l-6)$, $1\le i\le l$ \\
 & $4l$  & $A_{2l}(2l+j; 12l+6)\diamond A_{2l}(24l^2+12l+1-j; -12l-6)$, $1\le j\le 5l+3$ \\\hline
A-2 & $4l$ & $A_{2l}(i; 12l+2)\diamond A_{2l}(24l^2-2l-1-i; -12l-2)$, $1\le i\le 3l$  \\
  & $4l-2$ &  $A_{2l-1}(6l+j; 12l+2)\diamond A_{2l-1}(24l^2-8l-1-j; -12l-2)$, $1\le j\le 3l+1$ \\\hline
B-1 & $4l-2$ & $A_{2l-1}(i; 10l-1) \diamond A_{2l-1}(20l^2-12l+1-i; -10l+1)$, $1\le i\le 5l-1$ \\
 & $2l-2$ & $A_{l-1}(10l-1; 10l-1)\diamond A_{l-1}(20l^2-12l+1-i; -10l-1)$ \\\hline
B-2 & $4l+2$ & $A_{2l+1}(i; 10l+5)\diamond A_{2l+1}(20l^2+12l+1-i; -10l-5)$, $1\le i\le l$ \\
  & $4l$ & $A_{2l}(2l+j; 10l+5)\diamond A_{2l}(20l^2+10l+1-j; -10l-5)$, $1\le j\le 4l+2$ \\
 & $2l$ & $A_{l}(6l+3;10l+5)\diamond A_{l}(20l^2+6l-2; -10l+5)$ \\\hline
C & $4l+2$ & $A_{2l+1}(i; 8l+4)\diamond A_{2l+1}(16l^2+10l+1-i; -8l-4)$, $1\le i\le l$ \\
 & $4l$ & $A_{2l}(2l+j; 8l+4)\diamond A_{2l}(16l^2+8l+1-j; -8l-4)$, $1\le j\le 3l+2$ \\\hline
D-1 & $4l-2$ &  $A_{2l-1}(i; 6l-3)\diamond A_{2l-1}((2l-1)(6l-5)-i; -6l+3)$, $1\le i\le l-1$\\
 & $4l-4$ & $A_{2l-2}(2l-2+j; 6l-3)\diamond A_{2l-2}((2l-1)(6l-5)-2l+2-j; -6l+3)$, $1\le j\le 2l-1$ \\
 & $2l-2$ & $A_{l-1}(4l-2; 6l-3)\diamond A_{l-1}(6l-7)(2l-1); -6l+3)$ \\\hline
D-2 & $4l-2$  & $A_{2l-1}(i; 6l-1)\diamond A_{2l-1}((2l-1)(6l-1)-i; -6l+1)$, $1\le i\le 3l-1$ \\
& $2l-2$ & $A_{l-1}(6l-1; 6l-1)\diamond A_{l-1}((2l-2)(6l-1); -6l+1)$ \\\hline
E & $4l-2$ & $A_{2l-1}(i; 4l-2)\diamond A_{2l-1}(8l^2-10l+3-i; -4l+2)$, $1\le i\le l-1$ \\
  & $4l-4$ & $A_{2l-2}(2l-2+j; 4l-2) \diamond A_{2l-2}(8l^2-12l+5-j; -4l+2)$, $1\le j\le l$  \\\hline
F & $4s-10$ & $A_{2s-5}(i; 2s-5) \diamond A_{2s-5}((2s-5)^2-i; -2s+5)$, $1\le i\le s-3$  \\
  & $2s-6$ & $A_{s-3}(2s-5; 2s-5) \diamond A_{s-3}((2s-6)(2s-5); -2s+5)$ \\\hline
\end{tabular}
\caption{Edge labels of each induced path.}\label{table-2}
\end{table}

\end{theorem}

\section{Necessity Proof of Theorem~\ref{thm-iff} }\label{necessity}

\begin{proof}  Let $G\in \widetilde{\mathcal B}_3$ with size $m$. Keep all notation defined in  Sections 1 and 2. Throughout the following proof, we let $u, v, w$ be the three vertices of degree at least 3. Without loss of generality, we may assume that $G$ has $s$ distinct induced paths or cycles. Since $2s=\deg(u)+\deg(v)+\deg(w) \ge 3d=9$, $s\ge 5$.
Thus, by Lemma~\ref{lem-induced}, $m\ge 12$. Moreover, $G$ has order $m-s+3$.

\ms\nt From Lemma~\ref{lem-induced}, we only need to consider $a\in \{2,4,6,8,10,12\}$. For each of the $a$, we shall in what follows, determine the exact values of $s$, $x$, $y$ and $m$ that give a unique way to label the edges of the induced paths bijectively such that sum of all the end-edge labels is $3y$ and all the degree two vertices have induced vertex labels $x,y,\ldots,x$.

\ms\nt The following claim is given in the proof of~\cite[Theorem 2.1]{Lau+Shiu+Nal+Zhang+Prem}. By symmetry, we always assume $\a_1 < \b_r$.

\ms\nt{\bf Claim:} {\it  Let $\phi$ be a labeling of a path $P_{2r+1}= v_1v_2\cdots v_{2r+1}$ with $\phi(v_{2i-1}v_{2i}) = \a_i$ and $\phi(v_{2i}v_{2i+1}) = \b_i$ for $1\le i\le r$. Suppose $\phi^+(v_{2j}) = x$ for $1\le j\le r$ and $\phi^+(v_{2k+1}) = y$ for $0\le k\le r$, where $y > x$, then $\a_1+\b_1=x$, $\{\a_1, \a_2, \ldots, \a_r\}$ is an increasing sequence with common difference $y-x$ while $\{\b_1, \b_2, \ldots, \b_r\}$ is a decreasing sequence with common difference $y-x$.  Moreover, these sequences are determined uniquely when the initial value $\a_1$ or $\b_r$ is given.}

\ms \nt Throughout this paper, we shall always use $\a_1$ and $\b_{r}$ as end-edge labels for a path of length $2r$, $r\ge 1$.

\ms\nt We shall consider 6 cases for $a=12,10,8,6,4,2$, respectively. Recall that $m$ is even and $m\ge 12$.

\ms\nt {\bf Case~A. }
 $a = 12$.  By \eqref{eq-y}, we have $m=\frac{1}{3}(2s^2-5s-3)$ and  $y = \frac{1}{3}(2s^2+s)$. Hence, $x=\frac{1}{3}(2s^2-5s)$.
 Since $m$ is even, $2s^2-5s-3\equiv 0 \pmod 6$. This is equivalent to $2s^2-5s-3\equiv 0 \pmod 3$ and $2s^2-5s-3\equiv 0 \pmod 2$. Thus,
 $s(s-1)\equiv 0 \pmod 3$ and $s\equiv 1 \pmod 2$. So we have $s\equiv 1, 3\pmod 6$.

\ms\nt Since $y-x = 2s$,  all integers in $[1,2s]$ are end-edge labels. Since there are $s$ induced paths, all end-edges are in $[1, 2s]$. Let $P$ be an induced path of length $2r$ with end-edge labels $\a_1  < 2s$ and $\b_r = \b_1 - (r-1)(y-x) = x-\a_1 - 2rs + 2s \le 2s$. So, \[2r \ge \frac{x-\a_1}{s} >  \frac{(2s^2-5s)/3 - 2s}{s} = \frac{2s-11}{3}.\] Note that $2r$ is even. Since $\b_r\ge 2$, we have $2r \le \frac{1}{s}(x - \a_1 + 2s - 2) < \frac{2s+1}{3}$. So
\begin{equation}
\frac{2s-11}{6} <  r <\frac{2s+1}{6}.\label{eq-rA}\end{equation}
Consider the following two subcases.

\begin{enumerate}[{A-}1.]
\item Suppose  $s=6l+3$ for some $l\ge 1$, then $m = 24l^2+14l$, $y=24l^2+26l+7$ and $x=24l^2+14l+1$.  From \eqref{eq-rA}, we have $2l-\frac{5}{6} <  r<2l+\frac{7}{6}$, i.e., $r=2l$ or $r=2l+1$.  Suppose $G$ has $h$ induced path(s) of length $4l+2$ and $6l+3-h$ induced path(s) of length $4l$. We now have $(4l+2)h + 4l(6l+3-h) = m = 24l^2+14l$. Therefore, $G$ has $h = l$ induced path(s) of length $4l+2$ and $5l+3$ induced paths of length $4l$.

Suppose an induced path of length $4l+2$ has $1\le \a_1=i < \b_{2l+1} = \b_1 - 2l(12l+6) = x-i - 2l(12l+6) = 24l^2+14l+1-i - 2l(12l+6) = 2l+1-i$. Thus, $i < (2l+1)/2$, i.e., $1\le i\le l$. Without loss of generality, we have the $i$-th induced path has consecutive edge labels $A_{2l+1}(i;12l+6)\diamond A_{2l+1}(24l^2+14l+1-i; -12l-6)$ in order. Note that as a set, $A_{2l+1}(24l^2+14l+1-i;-12l-6) = A_{2l+1}(2l+1-i;12l+6)$. Now, integers in $[1,2l]$ are end-edge labels and all the integers used are in $\bigcup^{2l}_{k=0} [(12l+6)k+1,(12l+6)k+2l]$. Note that, Lemma~\ref{lem-AP} guarantees that all sequences are disjoint. we will not mention this fact again in the following discussion.

\ms\nt Thus, an induced path of length $4l$ has $2l+1\le \a_1 = 2l+j < \b_{2l} = 24l^2+14l+1-(2l+j)-(2l-1)(12l+6) = 12l+7-j$. Thus, $2j < 10l +7$, i.e., $1\le j \le 5l + 3$. Without loss of generality, we have the $j$-th induced path has consecutive edge labels $A_{2l}(2l+j;12l+6)\diamond A_{2l}(24l^2+12l+1-j; -12l-6)$ in order. Note that as a set, $A_{2l}(24l^2+12l+1-j; -12l-6) = A_{2l}(12l+7-j; 12l+6)$. Now, integers in $[2l+1,12l+6]$ are end-edge labels and all the integers used are in $\bigcup^{2l-1}_{k=0}[(12l+6)k+2l+1,(12l+6)k+12l+6]$. Thus, we have obtained a bijective edge labeling of the paths.

\item Suppose  $s=6l+1$ for some $l\ge 1$. From \eqref{eq-rA}, $2l-\frac{3}{2}  < r < 2l+\frac{1}{2}$, i.e., $r=2l-1$ or $r=2l$. Similar to Subcase~A-1, $G$ has $3l$ induced paths of length $4l$, and $3l+1$ induced paths of length $4l-2$. Moreover, $y= (6l+1)(4l+1)$ and $x= 24l^2-2l-1$.

\ms\nt Suppose an induced path of length $4l$ has $1\le \a_1=i < \b_{2l} = 24l^2-2l-1-i-(2l-1)(12l+2) = 6l+1-i$. Thus, $i < (6l+1)/2$, i.e., $1\le i\le 3l$. Without loss of generality, we have the $i$-th induced path has consecutive edge labels $A_{2l}(i;12l+2)\diamond A_{2l}(24l^2-2l-1-i;-12l-2)$ in order. Note that as a set, $A_{2l}(24l^2-2l-1-i;12l+2) = A_{2l}(6l+1-i;12l+2)$. Now, integers in $[1,6l]$ are end-edge labels and all the integers used are in $\bigcup^{2l-1}_{k=0} [(12l+2)k+1,(12l+2)k+6l]$.

\ms\nt Thus, an induced path of length $4l-2$ has $6l+1\le \a_1 = 6l+j < \b_{2l-1} = 24l^2-2l-1-(6l+j) - (2l-2)(12l+2) = 12l +3-j$. Thus, $j <(6l+3)/2$, i.e., $1\le j\le 3l+1$. Without loss of generality, we have the $j$-th induced path has consecutive edge labels $A_{2l-1}(6l+j;12l+2)\diamond A_{2l-1}(24l^2-8l-1-j; -12l-2)$ in order. Note that as a set $A_{2l-1}(24l^2-8l-1-j; -12l-2) = A_{2l-1}(12l +3-j; 12l+2)$. Now, integers in $[6l+1,12l+2]$ are end-edge labels and all the integers used are in $\bigcup^{2l-2}_{k=0} [(12l+2)k+6l+1, (12l+2)k+12l+2]$.

\ms\nt Therefore, we have obtained a bijective edge labeling of the paths.
\end{enumerate}

\ms\nt {\bf Case~B. } $a = 10$. Now, $b=\frac{4}{5}(2s^2-11s+15)$. Since $b\ge a$, $2s^2-11s+15\ge \frac{25}{2}$. This implies that $s\ge 6$. By~\eqref{eq-y}, we have $y = \frac{4s^2-2s}{5}$, $m = \frac{4s^2-12s}{5}$ and hence $x = \frac{4s^2-12s+5}{5}$. Thus, $4s^2-2s\equiv 0\pmod 5$. This implies that $s^2\equiv 3s\pmod 5$ or equivalently $s\equiv 0, 3\pmod{5}$. Thus $s\ge 8$.

\ms\nt Now $y-x = 2s-1$. Since $y$ is even, $y/2=(2s^2-s)/5$ is assigned to an end-edge. Since $s\ge 8$, $(2s^2-s)/5>2s-1$. So all integers in $[1,2s-1] \cup \{y/2 = (2s^2-s)/5\}$ must be assigned to end-edges. Since there are $s$ induced paths, all end-edges are $[1, 2s-1]\cup \{(2s^2-s)/5\}$.
Thus, there are $s-1$ induced paths have both end-edge labels in $[1, 2s-1]$. Let $P_{2r+1}$ be one of these paths. Since $\a_1 < \b_r$, we have $\a_1 \in [1,2s-2]$. Now, $\b_r = x - \a_1 - (r-1)(y-x) \le 2s-1 = y-x$. Since $x = \frac{4s^2-12s+5}{5}$ and $\a_1\le 2s-2$, we have
\[\frac{(2s-10)(2s-1)}{5} +1 = \frac{4s^2-22s+15}{5}\le x - \a_1 \le r(y-x) =  r(2s-1).\]
Thus, $r > \frac{2s-10}{5} > 1$, i.e., $r\ge 2$. Hence, $\b_{r-1}$ is a non-end-edge label so that $\b_{r-1} = x-\a_1 - (r-2)(y-x)\ge 2s$. Therefore
\[(r-2)(2s-1) \le x - \a_1 - 2s \le  (4s^2-22s)/5 = (2s-10)(2s-1)/5 - 2 < (2s-10)(2s-1)/5.\]
Consequently, $r-2 < (2s-10)/5$. Combining the above, we have
\begin{equation} 2s/5 - 2 < r < 2s/5.\label{eq-rB}\end{equation}
\begin{enumerate}[{B-}1.]
\item Suppose $s=5l$ for some $l\ge 2$. From \eqref{eq-rB}, $r=2l-1$. Thus there are $s-1=5l-1$ induced paths of length $4l-2$. Therefore, the remaining path must have length $m - (s-1)(4s/5 - 2) = 2l - 2$.

Note that $y=20l^2-2l$, $x=20l^2-12l+1$ and $y-x=10l-1$. Suppose an induced path of length $4l-2$ has $1\le \a_1=i <\b_{2l-1} =  20l^2-12l+1-i-(2l-2)(10l-1) = 10l-1-i$. Thus, $1\le i \le 5l-1$. Without loss of generality, we have the $i$-th induced path has consecutive edge labels $A_{2l-1}(i; 10l-1)\diamond A_{2l-1}(20l^2-12l+1-i;-10l+1)$ in order. Note that as a set, $A_{2l-1}(20l^2-12l+1-i;-10l+1)=A_{2l-1}(10l-1-i;10l-1)$. Now, integers in $[1,10l-2]$ are end-edge labels and all the integers used are in $\bigcup^{2l-2}_{k=0}[(10l-1)k+1,(10l-1)k+10l-2]$.

Consequently, the induced path of length $2l-2$ has consecutive edge labels $A_{l-1}(10l-1;10l-1)\diamond A_{l-1}(20l^2-22l+2;-10l+1)$. As a set, $A_{l-1}(20l^2-22l+2;-10l+1)=A_{l-1}(10l^2-l;10l-1)$. Thus, the integers used are in $\bigcup^{l-2}_{k=0}\{(10l-1)k+10l-1,(10l-1)k+10l^2-l\}$.

Thus, we have obtained a bijective edge labeling of the paths.

\item Suppose $s=5l+3$ for some $l\ge 1$. Now, $y=20l^2+22l+6$, $x=20l^2+12l+1$ and  $y-x=10l+5$.

Let $P$ be the induced path of length $2t$ with end-edge label $y/2=10l^2+11l+3$.   Therefore, another end-edge label is $\alpha_1\in[1, 2s-1]=[1, 10l+5]$. Thus $\beta_t=10l^2+11l+3=x-\alpha_1 -(t-1)(10l+5)$.
    So we have $1\le \alpha_1=10l^2+l-2-(t-1)(10l+5)\le 10l+5$. By solving this inequality, we have $t=l$ only. Hence, we have $\alpha_1=6l+3$. This path has consecutive edge labels $A_{l}(6l+3; 10l+5)\diamond A_{l}(20l^2+6l-2;-10l-5)$ in order. As a set,
    $A_{l}(20l^2+6l-2; -10l-5)=A_{l}(10l^2+11l+3; 10l+5)$.  Now, all the integers used are in\\ $\{(10l+5)k+6l+3, (10l+5)k+10l^2+11l+3\;|\; 0\le k\le l-1\}=\{(10l+5)k+6l+3\;|\; 0\le k\le 2l-1\}$.
Thus the end-edge labels of other induced paths are in $[1, 10l+5]\setminus\{6l+3\}$.

From \eqref{eq-rB}, we have $r=2l$ or $r=2l+1$. Let $h$ be the number of paths of length $4l$. Then there are $5l+2-h$ induced paths of length $4l+2$.
Thus, $m=(2l)+h(4l)+(5l+2-h)(4l+2)$ that gives $h=4l+2$. Hence, we have one induced path of length $2l$, $4l+2$ induced paths of length $4l$ and $l$ induced paths of length $4l+2$.

\nt
Suppose an induced path of length $4l+2$ has $1\le \a_1=i <\b_{2l+1} = 20l^2+12l+1-i-(2l)(10l+5) = 2l+1-i$. Thus, $1\le i \le l$.
Without loss of generality, we have the $i$-th induced path has consecutive edge labels $A_{2l+1}(i; 10l+5)\diamond A_{2l+1}(20l^2+12l+1-i;-10l-5)$ in order. Note that as a set, $A_{2l+1}(20l^2+12l+1-i;-10l-5)=A_{2l+1}(2l+1-i;10l+5)$.  Now, integers in $[1,2l]$ are end-edge labels and all the integers used are in $\bigcup^{2l}_{k=0}[(10l+5)k+1,(10l+5)k+2l]$.

Suppose an induced path of length $4l$ has $2l+1\le \a_1= 2l+j <\b_{2l} = 20l^2+12l+1-(2l+j)-(2l-1)(10l+5) = 10l+6-j$. Thus, $j < 4l+3$.
Without loss of generality, we have the $j$-th induced path has consecutive edge labels $A_{2l}(2l+j; 10l+5)\diamond A_{2l}(20l^2+10l+1-j;-10l-5)$ in order, where $1\le j\le 4l+2$. Note that as a set, $A_{2l}(20l^2+10l+1-j;-10l-5)=A_{2l}(10l+6-j;10l+5)$. Now, integers in $[2l+1,10l+5]\setminus\{6l+3\}$ are end-edge labels and all the integers used are in $\bigcup^{2l-1}_{k=0}\big\{[(10l+5)k+2l+1,(10l+5)k+10l+5]\setminus\{(10l+5)k+6l+3\}\big\}$.

Thus, we have obtained a bijective edge labeling of the paths.
\end{enumerate}

\ms\nt {\bf Case~C. } $a=8$.  Now, $b = 2s^2-11s+15$. By \eqref{eq-y}, we have $y = \frac{2s^2-3s+1}{2}$. Hence, $s$ is odd. Let $s=2k+1$. By \eqref{eq-y}, we have $m=4k^2-3k-1$. Since $m$ is even, $k=2l+1$ for some $l$. So $s=4l+3$. Note that, since $s\ge 5$, $l\ge 1$. Hence $m=16l^2+10l$, $y=(2l+1)(8l+5)$, $x=(2l+1)(8l+1)$.
Thus, all integers in $[1,8l+4]$ are end-edge labels. Note that $1+2+\cdots + (8l+4)  = 2y$. Thus, two other end-edge labels, say $\g_1$ and $\g_2$, must sum to $y = \g_1 + \g_2$.

\ms\nt Suppose $\g_1$ and $\g_2$ are end-edge labels of the same path of length $2q$. Without loss of generality, assume $\a_1=\g_1$ and $\b_q=\g_2$ so that $y=\a_1+\b_q = \a_1+(x-\g_1)-(q-1)(y-x)$ which implies that $q=0$, a contradiction. Thus, $\g_1$ and $\g_2$ are labeled at distinct  induced paths. Therefore, there are $s-2=4l+1$ paths with both end-edge labels in  $D'=[1,8l+4]$ and exactly two  induced paths, say $Q_i$, with an end-edge label in $D'$ and another end-edge label   $\g_i > 8l+4$, $i=1,2$.

\ms\nt  Let $u_i$ be the end vertex of $Q_i$ incident to the edge with label $\g_i$, $i=1,2$. By identifying $u_1$ with $u_2$, we get a path,  denoted $Q_1Q_2$, with end-edge labels in $[1,8l+4]$.  Note that, since $\g_1+\g_2=y$,  for  every internal vertex of $Q_1Q_2$, the induced vertex labels are $x$ and $y$ alternative like all other induced paths. Whenever necessary, we shall call it an {\it artificial induced path}.  Otherwise, we shall still call it an induced path for simplicity.  Thus, artificially $G$ contains $4l+2$ induced paths with end-edge labels in $D'$.

%Thus, the length of $Q_1Q_2$ must be $4l$ or $4l+2$.

%{\Shiu Thus, $G$ contains $4l+1$ induced paths with end-edge labels in $D'$ and two induced paths with an end-edge label in $D'$ and the other is not, for each.}

\ms\nt  Let $P_{2r+1}$ be a path with both end-edge labels in $[1,8l+4]$. Since $1\le \a_1 < \b_r \le 8l+4=y-x$, we have $\b_r = x-\a_1 - (r-1)(y-x) \le y-x$. So,
\[(2l+1)(8l+1)-(8l+4) <  x-\a_1 \le r(y-x) = r(8l+4)=4r(2l+1).\]
Thus, $r>2l-\frac{3}{4}$, i.e., $r\ge 2l$.

\nt Since $\b_{r-1}$ is a non-end-edge label, we have $\b_{r-1} = (x-\a_1)-(r-2)(y-x) \ge 8l+5$. Now,
\[(r-2)(8l+4) \le x-\a_1 - 8l-5 \le (2l+1)(8l+1)-1- 8l-5=16l^2+2l-5.\]
Thus $r-2\le 2l-\frac{6l+5}{8l+4}<2l$.  Therefore, $r=2l$ or $r=2l+1$.
%
%
%\ms\nt Assume $Q_1Q_2$ has length $4l$, and there are $h$ paths of length $4l+2$, then $4l + h(4l+2)+(4l+1-h)(4l) = m= 16l^2+10l$ so that  $h = l$. Therefore, in this case, there are $3l+1$ paths of length $4l$ and $Q_1Q_2$ is a path of length $4l$. Overall, there are $l$ path(s) of length $4l+2$, and $3l+2$ of length $4l$ with one of it will be split appropriately to get a $Q_1$ and a $Q_2$.
%
%\ms\nt Assume $Q_1Q_2$ has length $4l+2$, and there are $h'$ paths of length $4l$, then $(4l+2)+h'(4l) + (4l+1-h')(4l+2)=16l^2+10l$ so that $h'=3l+2$. Therefore, in this case, there are $3l+2$ path of length $4l$, and $l-1$ paths of length $4l+2$ and $Q_1Q_2$ is a path of length $4l+2$. Overall, there are $3l+2$ paths of length $4l$, and $l$ path(s) of length $4l+2$ with one of it will be split appropriately to get a $Q_1$ and a $Q_2$.

%\ms\nt Combining the above, we may begin with $l$ induced path(s) of length $4l+2$ and $3l+2$ induced paths of length $4l$ such that one of them must be split appropriately to get a $Q_1$ and a $Q_2$. Now, all end-edge labels of these $4l+2$ paths are in $[1, 8l+4]$ and vice versa.

\ms\nt Suppose an induced path of length $4l+2$ has $1\le \a_1=i <\b_{2l+1} = 16l^2+10l+1 - i - (2l)(8l+4) = 2l+1-i $. Thus, $2i < 2l+1$. Hence $1\le i\le l$.  Without loss of generality, the $i$-th path has consecutive edge labels $A_{2l+1}(i; 8l+4)\diamond A_{2l+1}(16l^2+10l+1 - i; -8l-4)$. Note that as a set, $A_{2l+1}(16l^2+10l+1 - i; -8l-4) = A_{2l+1}(2l+1 -i ; 8l+4)$. Now, integers in $[1,2l]$ are end-edge labels and all the integers used are in $\bigcup^{2l}_{k=0}[(8l+4)k+1,(8l+4)k+2l]$.

\ms\nt Thus, an induced path of length $4l$ has $2l+1\le \a_1 = 2l+j < \b_{2l} = 16l^2+10l+1 - (2l+j) - (2l-1)(8l+4) = 8l + 5 - j$. Thus, $j < (6l+5)/2$. Hence $1\le j\le 3l+2$. Without loss of generality, the $j$-th induced path has consecutive edge labels $A_{2l}(2l+j; 8l+4) \diamond A_{2l}(16l^2+8l+1 - j; -8l-4)$. Note that as a set,  $A_{2l}(16l^2+8l+1 - j; -8l-4)=A_{2l}(8l+5-j; 8l+4)$.  Now, integers in $[2l+1,8l+4]$ are end-edge labels and all the integers used are in $\bigcup^{2l-1}_{k=0}[(8l+4)k+2l+1,(8l+4)k+8l+4]$.

\ms\nt Therefore, we have obtained a bijective edge labeling of the paths.  Consequently, $G$ contains $l$ induced paths of length $4l+2$ and $3l+2$ induced paths of length $4l$. One of these paths is  the artificial induced path $Q_1Q_2$.

\ms\nt {\bf Case~D. } $a=6$. By \eqref{eq-y}, we have $y = \frac{4s^2-10s+6}{3}$, $m=\frac{4s^2-16s+12}{3}$ and $x=\frac{4s^2-16s+15}{3}$. Hence, $s\equiv 0, 1\pmod{3}$. It is easy to verify that $m$ and $y$ are even. Now $y-x = 2s-3$. Clearly, $y/2=\frac{2s^2-5s+3}{3}\ge 2s-2$. Thus, all the integers in $D'=[1,2s-3]\cup\{\frac{2s^2-5s+3}{3}\}$ are end-edge labels. Note that $1+2+\cdots + (2s-3) + y/2 = (8s^2-20s+12)/3 = 2y$. Thus, two other end-edge labels, say $\g_1$ and $\g_2$, must sum to $y=\g_1+\g_2$.
By an argument similar to that in Case~C, we know $\g_1$ and $\g_2$ are labeled at distinct paths. Let $Q_i$ be such an induced path with end-edge label $\g_i$, $i=1,2$. Therefore, another end-edge label of $Q_i, i=1,2,$ must be in $D'$. Since $\g_1+\g_2=y$, similar to Case~C, $Q_1Q_2$ is an induced path with end-edge labels in $D'$.
Consequently, $G$ is constructed by using $s-1$ induced paths with end-edge labels in $D'$ in which one of them is an artificial induced path that must be split appropriately to get $Q_1,Q_2$.

\ms\nt Let $P_{2r+1}$ be a path with both end-edge labels in $[1,2s-3]$. Since $1\le \a_1 < \b_r \le 2s-3$, we have $\b_r = x-\a_1 - (r-1)(y-x) \le 2s-3$. So,
\[\frac{(2s-3)(2s-8)}{3}=\frac{4s^2-22s+24}{3}=x-(2s-3)< x-\a_1\le r(2s-3).\]
Thus $r> \frac{2s-8}{3}$. Since $\b_{r-1}$ is a non-end-edge label, we have $\b_{r-1} = (x-\a_1) - (r-2)(y-x) \ge 2s-2$. Now \[(r-2)(2s-3) \le x-\a_1 - 2s+2\le \frac{4s^2-22s+18}{3}<\frac{(2s-3)(2s-8)}{3}.\] Thus, $r-2< \frac{2s-8}{3}$. Consequently, $\frac{2s-8}{3}<r<\frac{2s-2}{3}$.
Suppose $s=3l\ge 6$. We have $r=2l-2, 2l-1$. Suppose $s=3l+1\ge 7$. We have $r=2l-1$.

\begin{enumerate}[D-1.]
\item Suppose $s=3l$, where $l\ge 2$. Note that, $m=4(3l-1)(l-1)$, $x=(2l-1)(6l-5)$, $y=2(3l-1)(2l-1)$, all the integers in $[1, 6l-3]\cup\{(3l-1)(2l-1)\}$ are end-edge labels.

\begin{enumerate}[a)]
\item Suppose an induced path of length $4l-2$ has  $1\le \a_1 < \b_{2l-1} = (2l-1)(6l-5) - \a_1 - (2l-2)(6l-3) = 2l-1-\a_1$. Thus, $2\a_1 < 2l-1$, i.e., $1\le \a_1\le l-1$. So we get that $\a_1\in [1, l-1]$ and $\b_{2l-1}\in [l, 2l-2]$. Hence there are at most $l-1$ induced paths of length $4l-2$.

\item Suppose an induced path of length $4l-4$ has $\b_{2l-2}\le 6l-3$. Therefore, $\a_1=x-(2l-3)(6l-3)-\b_{2l-2}=8l-4-\b_{2l-2}\ge 8l-4-6l+3=2l-1$. On the other hand, $\a_1<\b_{2l-2}=8l-4-\a_1$. Thus $\a_1<4l-2$, i.e., $\a_1\le 4l-3$.

    So, $\a_1\in[2l-1, 4l-3]$ and hence $\b_{2l-2}\in[4l-1, 6l-3]$. Hence there are at most $2l-1$ induced paths of length $4l-4$.
\end{enumerate}
\nt Thus, there is another path whose end-edge labels are $4l-2$ and $(3l-1)(2l-1)$. We can obtain that this path is of length $2l-2$ and the sequence of edge labels is $A_{l-1}(4l-2; 6l-3)\diamond A_{l-1}((6l-7)(2l-1); -6l+3)$.  Note that as a set, $A_{l-1}((6l-7)(2l-1); -6l+3) = A_{l-1}((3l-1)(2l-1);6l-3)$. Now, all the integers used are in $\{3(2l-1)k+(4l-2) \mid 0\le k\le 2l-3\}$.

\nt Since there are $3l-1$ paths under our consideration, there are exactly $l-1$ paths of length $4l-2$, $2l-1$ paths of length $4l-4$ and one path of length $2l-2$. Thus, we need to split one of the above paths appropriately to get $Q_1$ and $Q_2$. Without loss of generality, the $i$-th induced path, $1\le i\le l-1$, of length $4l-2$ has consecutive edge labels $A_{2l-1}(i; 6l-3) \diamond A_{2l-1}((2l-1)(6l-5)-i; -6l+3)$. Note that as a set, $A_{2l-1}((2l-1)(6l-5)-i; -6l+3) = A_{2l-1}(2l-1-i; 6l-3)$. Now, all the integers used are in $\bigcup^{2l-2}_{k=0}[(6l-3)k+1, (6l-3)k+2l-2]$. Moreover, the $j$-th induced path, $1\le j\le 2l-1$, of length $4l-4$ has consecutive edge labels in $A_{2l-2}(2l-2+j; 6l-3) \diamond A_{2l-2}((2l-1)(6l-5)-2l+2-j; -6l+3)$. Note that as a set, $A_{2l-2}((2l-1)(6l-5)-2l+2-j; -6l+3) = A_{2l-2}(6l-2-j; 6l-3)$. Now, all the integers used are in $\bigcup^{2l-3}_{k=0}[(6l-3)k+2l-1, (6l-3)k+4l-3] \cup \bigcup^{2l-3}_{k=0}[(6l-3)k+4l-1, (6l-3)k+6l-3]$. Together with the edge labels of the induced path of length   $2l-2$, we have obtained a bijective edge labeling of the $3l-1$ paths.

\item Suppose $s=3l+1$, where $l\ge 2$. Note that, $m=4l(3l-2)$, $x=(2l-1)(6l-1)$, $y=2l(6l-1)$, all integers in $[1, 6l-1]\cup\{l(6l-1)\}$ are end-edge labels. By a similar argument as Case~D-1, we obtain that there are at most $3l-1$ induced paths of length $4l-2$ with both end-edge labels $[1, 6l-2]$.  So there are exactly $3l-1$ induced paths of length $4l-2$ with both end-edge labels $[1, 6l-2]$. Thus, there is another path whose end-edge labels are $6l-1$ and $l(6l-1)$. We can obtain that this path is of length $2l-2$ and the sequence of edge labels  is  $A_{l-1}(6l-1;6l-1) \diamond A_{l-1}((2l-2)(6l-1); -6l+1)$.  Note that as a set, $A_{l-1}((2l-2)(6l-1); -6l+1) = A_{l-1}((l-1)(6l-1); 6l-1)$. Now, all the integers used are in $\{(6l-1)k+(6l-1) \mid 0\le k\le 2l-3\}$.

\nt Therefore, there are exactly $3l-1$ paths of length $4l-2$ and one path of length $2l-2$. Thus, we need to split one of the above paths appropriately to get $Q_1$ and $Q_2$. Without loss of generality, the $i$-th induced path, $1\le i\le 3l-1$, of length $4l-2$ has consecutive edge labels $A_{2l-1}(i; 6l-1) \diamond A_{2l-1}((2l-1)(6l-1)-i; -6l+1)$. Note that as a set, $A_{2l-1}((2l-1)(6l-1)-i; -6l+1) = A_{2l-1}(6l-1-i; 6l-1)$. Now, all the integers used are in $\bigcup^{2l-2}_{k=0}[(6l-1)k+1, (6l-1)k+6l-2]$. Together with the edge labels of the induced path of length $2l-2$, we have obtained a bijective edge labeling of the $3l$ paths.
\end{enumerate}

\ms\nt {\bf Case~E. } $a=4$.  By~\eqref{eq-y}, we have $y = 2s^2-7s+6$, $m=2s^2-9s+9$ and $x=2s^2-9s+10$. Hence, $y-x=2s-4$. Since $m$ is even, we have $s$ is odd. Thus, all integers in $[1,2s-4]$ are end-edge labels. Note that $1+2+\cdots+(2s-4) = y$. Thus, four other end-edge labels, say $\g_1, \g_2, \g_3, \g_4$, must sum to $2y$.

\ms \nt Now, let us consider induced paths with end-edge label $\g_i$, $1\le i\le 4$. Suppose there are two induced paths $R_1$ and $R_2$ with length $2r_1$ and $2r_2$, respectively, such that their end-edge labels are $\g_i$, $1\le i\le 4$.

\nt Without loss of generality, assume $\g_1$ and $\g_2$ are end-edge labels for $R_1$ and  $\g_3$ and $\g_4$ are end-edge labels for $R_2$, where $\g_1<\g_2$ and $\g_3<\g_4$. Thus,
\begin{align*}
\g_1+\g_2 & = \g_1+(x-\g_1)-(r_1-1)(y-x)=x-(r_1-1)(y-x)\le x\\
\g_3+\g_4 & = \g_3+(x-\g_3)-(r_2-1)(y-x)=x-(r_2-1)(y-x)\le x.
\end{align*}
Now we get $2y=\g_1+\g_2+\g_3+\g_4\le 2x$ which is impossible.
Thus, there is at most one induced path with both end-edge labels are $\g_i$'s.

\ms\nt Let $P_{2r+1}$ be an induced path with both end-edge labels in $[1,2s-4]$.  Since $1\le \a_1 < \b_r\le 2s-4$, we have $\b_r = x-\a_1 - (r-1)(y-x) \le 2s-4$. So, \[(2s-4)(s-4) =2s^2-12s +16 < 2s^2-11s+14 <  x-\a_1 \le r(2s-4).  \] Thus, $r > s-4$, i.e., $r\ge s-3$. Since $\b_{r-1}$ is a non-end-edge label, we have $\b_{r-1} = (x-\a_1) - (r-2)(y-x) \ge 2s-3$. Now, \[(r-2)(2s-4) \le x-\a_1-2s+3\le  2s^2-9s+10 - 1 -2s+3 = 2s^2 -11s + 12 < 2s^2-10s+12 = (2s-4)(s-3).\] Thus, $r < s-1$, i.e., $r\le s-2$. Therefore, $r\in \{s-2,s-3\}$. Write $s=2l+1$, $l\ge 2$.

\begin{enumerate}[a)]
\item Suppose an induced path of length $4l-2$ has $\a_1\ge 1$. Thus, $1\le \a_1 < \b_{2l-1} = 2(2l+1)^2-9(2l+1)+10 - \a_1 - (2l-2)(4l-2) = 2l-1-\a_1$. Thus, $2\a_1 < 2l-1$, i.e., $1\le \a_1\le l-1$. So we get that $\a_1\in [1, l-1]$ and $\b_{2l-1}\in [l, 2l-2]$. Hence there are at most $l-1$ induced paths of length $4l-2$.

   Without loss of generality, the $i$-th path has consecutive edge labels $A_{2l-1}(i; 4l-2) \diamond A_{2l-1}(8l^2-10l+3-i;-4l+2)$. Note that as a set, $A_{2l-1}(8l^2-10l+3-i; -4l+2) = A_{2l-1}(2l-1-i; 4l-2)$. Now, integers in $[1,2l-2]$ are end-edge labels and all the integers used are in $\bigcup^{2l-2}_{k=0} [(4l-2)k+1,(4l-2)k+2l-2]$.

\item Suppose an induced path of length $4l-4$ has $\b_{2l-2}\le 4l-2$. Thus, $\a_1 =x-(2l-3)(4l-2)-\b_{2l-2}=3(2l+1)-6-\b_{2l-2}\ge 6l-3-4l+2=2l-1$. Now, $\a_1<\b_{2l-2}=6l-3-\a_1$. Thus $2\a_1<6l-3$, i.e., $\a_1\le 3l-2$. So, $\a_1\in[2l-1, 3l-2]$ and hence $\b_{2l-2}\in[3l-1, 4l-2]$. Hence there are at most $l$ induced paths of length $4l-4$.

    Without loss of generality, the $j$-th path has consecutive edge labels $A_{2l-2}(2l-2+j; 4l-2)\diamond A_{2l-2}(8l^2-12l+5-j; -4l+2)$. Note that as a set, $A_{2l-2}(8l^2-12l+5-j; -4l+2) = A_{2l-2}(4l-1-j; 4l-2)$. Now, integers in $[2l-1,4l-2]$ are end-edge labels and all the integers used are in $\bigcup^{2l-3}_{k=0}[(4l-2)k+2l-1,(4l-2)k+4l-2]$.
\end{enumerate}
We have obtained a bijective edge labeling of $2l-1$ paths using all the integers in $[1,m]$. However, $G$ has $2l+1$ induced paths. So some of these paths are not induced paths of $G$. Therefore, one (or two) such path(s) must be split to become three (or four) induced paths (or cycles) of $G$. Let us consider the following two subcases.
\begin{enumerate}[E-1.]
\item Suppose there is no induced path with both end-edge labels are $\g_i$'s. Let $Q_i$ be such a path with length $2r_i$, $1\le i\le 4$. Note that, another end-edge label must lie in $D'=[1, 4l-2]$.

    By the uniqueness of the labeling of the path, $Q_i$ must be a proper subpath of an induced path, say $P$, with both end-edge labels in $D'$. Thus, the path $P-Q_i$ obtain by removing all edges of $Q_i$ and all isolated vertices of the residue subgraph, is another $Q_{j}$. Thus, we have two artificial induced paths. Hence, artificially $G$ contains $2l-1$ induced paths whose end-edge labels are in $D'$.

    Thus, $\g_1,\ldots,\g_4$ must be end-edge labels of 4 paths obtained by splitting two paths of length $4l-2$ or $4l-4$ into 2 paths each.
\item Suppose there is one induced path $Q$ with both end-edge labels are $\g_3, \g_4$, say. Let $Q_j$ be the induced paths with one end-edge label $\g_j$ and the other end-edge label lying in $D'$, $j=1,2$.

We have found the $2l-1$ induced paths that used up all the labels, yet we need two more induced paths. So these two must come from splitting an existing path, which is now an artificial induced path $P$, into 3 induced paths. Thus, this artificial induced path is of the form $Q_1QQ_2$. Without loss of generality, $\g_1$ and $\g_3$ (also, $\g_4$ and $\g_2$) are adjacent edge labels in $P$. So $\g_1+\g_3=\g_4+ \g_2=y$.
\end{enumerate}

\ms\nt {\bf Case~F. } $a=2$. By~\eqref{eq-y}, we have $y=4s^2-18s+20 = 2(s-2)(2s-5)$ is even, $m=4s^2-20s+24$ and $x=4s^2-20s+25=(2s-5)^2$. Hence, $y-x=2s-5$. Thus, all integers in $D'=[1,2s-5]\cup \{(s-2)(2s-5)\}$ are end-edge labels. Note that $1+2+\cdots+(2s-5) = y/2$. Thus, four other end-edge labels, say $\g_1, \g_2, \g_3, \g_4$, must sum to $2y$. By the same argument as in Case~E, there is at most one induced path with both end-edge labels are $\g_i$'s.

\begin{enumerate}[a)]
\item Let $P_{2r+1}$ be a path with both end-edge labels in $[1,2s-5]$. Since $1\le \a_1 < \b_r \le 2s-5$, we have $\b_r = x-\a_1-(r-1)(y-x) \le 2s-5$. So,
\[(2s-6)(2s-5) = 4s^2-22s+30 < x-\a_1 \le r(y-x) = r(2s-5). \] Thus, $r > 2s-6$. Since $\b_{r-1}$ is a non-end-edge label, we have $\b_{r-1} = (x-\a_1) - (r-2)(y-x) \ge 2s-4$. Now,
\[(r-2)(2s-5) \le x -\a_1-2s+4 \le  4s^2 - 22s + 28 < (2s-6)(2s-5).\] Thus, $r-2 < 2s-6$. Therefore, $r=2s-5$.

\ms\nt Suppose an induced path of length $4s-10$ has $1\le \a_1 = i < \b_{2s-5} = 4s^2-20s+25 - i-(2s-6)(2s-5) = 2s - 5 - i$. Thus, $i < (2s-5)/2$, i.e., $1\le i\le s-3$. So there are at most $s-3$ induced paths of length $4s-10$ with $\a_1\in[1, s-3]$ and $\b_{2s-5}\in[s-2, 2s-6]$.

Without loss of generality, the $i$-th induced path has consecutive edge labels $A_{2s-5}(i; 2s-5) \diamond A_{2s-5}((2s-5)^2-i; -2s+5)$, $1\le i\le s-3$. Note that as a set, $A_{2s-5}((2s-5)^2-i; -2s+5) = A_{2s-5}(2s-5-i; 2s-5)$.  All the integers used are in $\bigcup^{2s-6}_{k=0}[(2s-5)k+1, (2s-5)k+2s-6]$.

\item Suppose $T$ is a path with end-edge label $2s-5$. Since $2s-5=\min\{2s-5, (s-2)(2s-5), \g_1, \g_2, \g_3, \g_4\}$, $2s-5$ is the first edge label of $T$. By uniqueness, the sequence of edge labels of $T$ must be a subsequence of $A_{s-3}(2s-5; 2s-5)\diamond A_{s-3}((2s-6)(2s-5); -2s+5)$. Note that as a set, $A_{s-3}((2s-6)(2s-5); -2s+5) = A_{s-3}((s-2)(2s-5); 2s-5)$. All the integers used are in $\{(2s-5)(k+1) \;|\; 0\le k \le 2s-7\}$. Thus, the possible case is either there is a path, say $T_0$, of length $2s-6$ with edge labels sequence $A_{s-3}(2s-5; 2s-5)\diamond A_{s-3}((2s-6)(2s-5); -2s+5)$ or there are two paths, say $T_1$ and $T_2$, with an end-edge label $2s-5$ and $(s-2)(2s-5)$, respectively. For the last case, $T$ is split appropriately to get $T_1$ and $T_2$. So we may consider $T$ as an artificial induced path of length $2s-6$.
\end{enumerate}
Now, artificially $G$ contains $s-3$ induced paths of length $4s-10$ and one induced path of length $2s-6$. We have now obtained a bijective edge labeling of these $s-2$ paths.
\begin{enumerate}[F-1.]
\item Suppose there is no induced path with both end-edge labels are $\g_j$'s. Let $Q_j$ be such a path with length $2r_j$, $1\le j\le 4$.
    Note that, another end-edge label must lie in $D'$. By a similar argument as in Case~E, $Q_1, Q_2, Q_3, Q_4$ are obtained by splitting two paths of length $4s-10$ or $2s-6$ into two paths respectively.

\item Suppose there is one induced path $Q$ with both end-edge labels are $\g_3, \g_4$, say. Let $Q_j$ be the induced paths with one end-edge label $\g_j$ and the other end-edge label in $D'$, $j=1,2$. By a similar argument as in Case~E, $Q_1, Q_2, Q$ are obtained by splitting a paths of length $4s-10$ or $2s-6$ into three paths, where $Q$ is the middle path.
\end{enumerate}
From the above, we have Table~\ref{table-1} and Table~\ref{table-2}.
\end{proof}

\section{Sufficiency Proof of Theorem~\ref{thm-iff}}

\nt  We keep all notation defined in Sections~1, 2 and 3. Suppose $P$ is an induced $(v_1,v_2)$-path of length $2r$. Recall that, we have used $\a_1$ and $\b_r$ to denote its end-edge label incident to $v_1$ and $v_2$, respectively, where $\a_1<\b_r$.
We shall use an order pair $(\a_1, \b_r)_{2r}$ to represent $P$. We also use $(a_1, a_2)$ to denote an induced path with end-edge labels $a_1, a_2$ without specifying  the  labels adjacent to $a_1,a_2$.

\ms\nt For a set $A$ of integers, we let $\sg(A)$ be the sum of elements in $A$, then
$\sg(D_u)=\sg(D_v)=\sg(D_w)=y$ and no two entries in the same partite set belong to length two induced path. Clearly, $\deg(u)=|D_u|$, $\deg(v)=|D_v|$ and $\deg(w)=|D_w|$.  Therefore, a solution  is obtained if and only if such a partition exists and a complete characterization is available if and only if all such partitions are found. However, it is well known that such a partitioning is an NP-hard problem. Therefore, in this section, we showed the necessary conditions are also sufficient.

\begin{lemma}\label{lem-split} Suppose $G\in \widetilde{\mathcal B}_3$. Let $D^*=\{x-\e \;|\;\e\in D'\}=\{m+1-\e\;|\; \e\in D'\}$. Suppose $\sg \in [1,m]\setminus (D'\cup D^*)$. There exists an induced path $P$ of length at least $6$ with an edge label $\sg$. Moreover, there is another label, say $\t$, that is adjacent to $\sg$ in $P$ such that
\begin{enumerate}[(a)]
\item $\sg+\t=y$,
\item $P$ can be split into two paths $Q_1$ and $Q_2$ of even length, and
\item $(Q_1, Q_2) =((\a_1, \sg), (\t,\b_r))$ or $((\a_1, \t), (\sg,\b_r))$.
\end{enumerate}

\nt Further, every induced path $P$ of length at least $8$ contains an edge with label $\sg$, and can be split into two paths $Q_1$ and $Q_2$ of even length at least $4$, with $(Q_1, Q_2) =((\a_1, \sg), (\t,\b_r))$ or $((\a_1, \t), (\sg,\b_r))$.
\end{lemma}

\begin{proof}  Let $P$ be an induced path of $G$ such that $\sg$ is labeled at $e\in E(P)$. Since $\sg \notin D\cup D^*$, it is not assigned to any end-edges and their neighbors. Thus, $e$ is adjacent to edges with label $y-\sg$ and $x-\sg$, respectively but not adjacent to any end-edges. Since the length of $P$ is even, $2r\ge 6$.   %  such that both $y-\sg$ and $x-\sg$ are not in $D' \cup D^*$.

\ms\nt Denote $y-\sg$ by $\t$, we have (a). Let $e'$ be the edge in $P$ with label $\t$. Let $z$ be the common vertex of $e$ and $e'$. Split $P$ into two paths $Q_1$ and $Q_2$ with $z$ as their common end-vertex. We have (b). Without loss of generality, we may assume $\a_1$ is an end-edge label of $Q_1$. The other end-edge label is either $\sg$ or $\t$. We have (c).

\ms\nt Further, if $P$ is of length $2r\ge 8$ with end-edge labels $\a_1,\b_r$, then we can let $e$ be an edge which is a distance at least three from both end-edges of $P$. Split $P$ into two paths $Q_1$ and $Q_2$ with $z$ as their common end-vertex, we have $Q_1$ and $Q_2$ are of even length at least 4, with $(Q_1, Q_2) =((\a_1, \sg), (\t,\b_r))$ or $((\a_1, \t), (\sg,\b_r))$.   \end{proof}

\begin{rem} When $s\ge 6$. $|[1,m]\setminus (D\cup D^*)|=m-2s$. From \eqref{eq-y}, we can see $m-2s\ge 1$.  From the proof of Theorem~\ref{thm-iff}, we know $s=5$ if and only if $a=2,4$. So by \eqref{eq-y} $m-2s=4s^2-22s+24=14$ or $m-2s = 2s^2-11s+9 = 4$. Thus $\sg$ always exists.
 Moreover,  if we choose $\sg \in D^*$, then $\t = y-\sg$ is a non-end-edge label. Thus,  either $Q_1$ or $Q_2$ is of length 2.
\end{rem}

\nt For graphs $G, H$, let $G+H$ denote the disjoint union of $G$ and $H$. For $c\ge 2$, $cG$ is the disjoint union of $c$ copies of $G$.  For $t\ge 2$ and $3\le a_1\le a_2\le \cdots\le a_t$, denote by $C(a_1,a_2,\ldots,a_t)$ the {\it one point union of $t$ cycles} of order $a_1, a_2,\ldots,a_t$, respectively \cite{LSN-IJMSI}.

\ms\nt We now present the sufficiency proof of Theorem~\ref{thm-iff}.

%\ms\nt Before showing some examples, we introduce some graphs and their corresponding notation.  A graph consisting of $t$ paths joining two vertices is called a {\it $t$-bridge graph}, which is denoted by $\th(a_1,\dots, a_t)$, where $t\ge 2$ and $1\le a_1\le a_2 \le \cdots \le a_t, a_2\ge 2$ are the lengths of the $t$ paths \cite{Lau+Shiu+Nal+Zhang+Prem, LSZPN}.  For $t\ge 2$ and $3\le a_1\le a_2\le \cdots\le a_t$, denote by $C(a_1,a_2,\ldots,a_t)$ the {\it one point union of $t$ cycles} of order $a_1, a_2,\ldots,a_t$, respectively \cite{LSN-IJMSI}.
\begin{proof}  Suffice to show the existence of a solution for all possible $s$ in each case of Section~3.

\ms\nt{\bf Case~A-1. }
In this case, we have $s=6l+3$, $l\ge 1$, $y=24l^2+26l+7=(2l+1)(12l+7)$ and $D'=[1, 12l+6]=D$.  There are $l$ induced paths of length $4l+2$, namely $(i, 2l+1-i)$, $1\le i\le l$; and $5l+3$ paths of length $4l$, namely $(2l+j, 12l+7-j)$, $1\le j\le 5l+3$.

\begin{example} Take the smallest $s=9$, i.e., $l=1$. We have $y = 57$, $x = 39$, $m=38$ and  $D=[1, 18]$. There is an induced path of length 6 and eight induced paths of length 4. Following are their representative order pairs and the corresponding consecutive edge labels:

\centerline{$\begin{array}{lll}(1,2): 1,38,19,20,37,2; & (3,18): 3,36,21,18;  & (4,17): 4,35,22,17;\\ (5,16): 5,34,23,16;  & (6,15): 6,33,24,15;  &  (7,14): 7,32,25,14; \\ (8,13): 8,31,26,13; & (9,12): 9,30,27,12; & (10,11): 10,29,28,11.\end{array}$}

\nt We may choose $D_u=\{1,5,11,7,15,18\}$,  $D_v=\{2,16,10,12,13,4\}$ and $D_w= \{14,6,3,9,8,17\}$. It is easy to see that $\sg(D_u)=\sg(D_v)=\sg(D_w)=57$. So it corresponds to a required solution $G$. Note that, in $G$, $\deg(u)=\deg(v)=\deg(w)=6$ and there are no induced cycles. \rsq
\end{example}

\ms\nt For a general $l\ge 1$, if we request $|D_u|=|D_v|=|D_w|=4l+2$, then we may let
\begin{align*} D_u & =\{1+6t, 6+6t \;|\; 0\le t\le 2l\}\\
D_v &= \{2+6t, 5+6t \;|\; 0\le t\le 2l\}\\
D_w &=\{3+6t, 4+6t\;|\; 0\le t\le 2l\}.\end{align*} Clearly,  $\sg(D_z) = (7+12l)(2l+1) = y$ for each $z\in\{u,v,w\}$.
Thus, we get a (possibly disconnected graph) solution such that vertices $u$, $v$, $w$ have equal degree $4l+2$.

\ms\nt{\bf Case A-2. }
Now $s=6l+1$, $l\ge 1$ and $y=(6l+1)(4l+1)$ with $D'=[1, 12l+2]=D$.  There are $3l$ induced paths of length $4l$,  namely $(i, 6l+1-i)$, $1\le i\le 3l$; and $3l+1$ induced paths of length $4l-2$, namely $(6l+j, 12l+3-j)$, $1\le j\le 3l+1$.
\begin{example} Suppose $l=1$, i.e., $s=7$.
There are four induced paths of length 2 and three induced paths of length 4. Their consecutive edge labels are given by

\centerline{$1,20,15,6;\quad  2,19,16,5; \quad  3,18,17,4;\quad  7,14; \quad 8,13; \quad 9,12;\quad 10,11.$}

\nt Note that integers $6+i$ and $15-i$ $(1\le i \le 4)$ must not be in the same partite set.  We let
$D_u= \{1,8,10,7,9\}$, $D_v=\{6,13,11,2,3\}$ and $D_w=\{14,12,5,4\}$. Hence, we get a required solution which is a biregular graph.\rsq
\end{example}

\ms\nt We now give three solutions for each $l\ge 2$. Let
\begin{align*}
D_u & = [1,l+1]\cup[5l,6l]\cup[6l+1,7l]\cup[11l+3,12l+2],\\
D_v & = [l+2,2l+2]\cup[4l-1,5l-1]\cup[7l+1,8l]\cup[10l+3,11l+2],\\
D_w & = [2l+3,4l-2]\cup[8l+1,10l+2].
\end{align*}
One may easy check that $\sg(D_z)=(6l+1)(4l+1)$, $z\in\{u,v,w\}$ and $|D_u|=|D_v|=2l+1$, $|D_w|=2l-1$. The corresponding graph is $2C((4l-2)^{[l]}, (4l)^{[l+1]})+C((4l-2)^{[l+1]}, (4l)^{[l-2]})$.

\ms\nt Suppose we swap the end-edge label $8l+1$ in $D_w$ with end-edge labels $4l$ and $4l+1$ in $D_v$. This gives the second solution which is $C((4l-2)^{[l]},(4l)^{[l+1]})+H$, where $H$ is a connected graph containing three induced $(v,w)$-paths, namely $(2l, 4l+1)_{4l}$, $(2l+1, 4l)_{4l}$, $(8l+1,10l+2)_{4l-2}$.

\ms\nt From the resulting partite sets, swap the end-edge label $10l+3$ in $D_v$ with end-edge labels $5l+1$ and $5l+2$ in $D_u$. This gives the third solution: a connected graph. Moreover, comparing with the second solution, this graph contains three more induced $(u,v)$-paths, namely $(l, 5l+1)_{4l}$, $(l-1, 5l+2)_{4l}$ and $(10l+3, 8l)_{4l-2}$.

\ms\nt{\bf Case~B-1. } Now, $s=5l$, $l\ge 2$ with  $D'=[1, 10l-1]\cup\{l(10l-1)\}=D$ and $y=2l(10l-1)$. There are $5l-1$ induced paths of length $4l-2$, namely $(i, 10l-1-i)$, $1\le i\le 5l-1$; and one induced path of length $2l-2$, namely $(10l-1, l(10l-1))$. %{\Lau [Looks like we are short of solution for $s=5$.]}

\begin{example}

\ms\nt Suppose $l=2$, i.e., $s=10$. We get $y = 76$ and the end-edge label set is $[1,19]\cup\{38\}$.
 Let  $D_u=\{1,4,10,11,12,38\}$, $D_v=\{7,9,13,14,16,17\}$, $D_w=\{2,3,5,6,8,15,18,19\}$. The corresponding graph is a 2-connected graph without induced cycle.

\ms\nt If we swap $1,4$ in $D_u$ with $5$ in $D_w$, then we get a 2-connected graph with two induced cycles incident to $w$, namely $(1,18)$ and $(4,15)$.
\rsq
\end{example}

\ms\nt We now give three solutions for each $l\ge 3$. Let
\begin{align*}
D_u & = [1,2l]\cup[8l-1,10l-2],\\
D_v & = [2l+1,4l]\cup [6l-1,8l-2],\\
D_w & = [4l+1,6l-2]\cup\{10l-1,10l^2-l\}.
\end{align*}
It is routine to check that $\sg(D_z)=l(20l-2)$, $z\in\{u,v,w\}$ and $|D_u|=|D_v|=4l$, $|D_w|=2l$.
The corresponding graph is $2C((4l-2)^{[2l]})+C((2l-2), (4l-2)^{[l-1]})$.

\ms\nt Suppose we swap  $10l-2$ in $D_u$ with $4l+1$ and $6l-3$ in $D_w$. This gives the second solution which is $H+C((4l-2)^{[2l]})$, where $H$ is a connected graph containing three induced $(w,u)$-paths, namely $(10l-2,1)_{4l-2}$, $(6l-2,4l+1)_{4l-2}$ and $(4l+2, 6l-3)_{4l-2}$.

\ms\nt From the resulting partite sets, swap $10l-3$ in $D_u$ with $2l+1$ and $8l-4$ in $D_v$. This gives the third solution: a connected graph. Moreover, comparing with the second solution, this graph contains three more induced $(u,v)$-paths, namely $(2,10l-3)_{4l-2}$, $(2l+1, 10l-2)_{4l-2}$ and $(8l-4, 2l+3)_{4l-2}$.

\ms\nt{\bf Case~B-2. } Now, $s=5l+3$, $l\ge 1$. We get  $y=2(5l+3)(2l+1)$ and  $D'=[1,10l+5]\cup\{(5l+3)(2l+1)\}=D$.
There are $l$ induced paths of length $4l+2$, namely $(i, 2l+1-i)$, $1\le i\le l$; $4l+2$ induced paths of length $4l$, namely $(2l+j, 10l+6-j)$, $1\le j\le 4l+2$; and one induced path of length $2l$, namely $(6l+3, 10l^2+11l+3)$.

\begin{example}
\ms\nt Suppose $l=1$, i.e., $s=8$. We get $y=48$. The edge labels of the paths are
\[\begin{array}{llll}
1,32,16,17,31,2; & 3,30,18,15; & 4,29,19,14; & 5,28,20,13; \\
6,27,21,12; & 7,26,22,11; & 8,25,23,10; & 9,24.
\end{array}\]

\nt We can have $D_u = \{1,4,5,9,14,15\}$, $D_v=\{2,10,11,12,13\}$, $D_w =\{3,6,7,8,24\}$ or else $D_u = \{1,4,5,6,8,9,15\}$, $D_v = \{2,10,11,12,13\}$, $D_w = \{3,7,14,24\}$. The first solution contains one $(u,u)$-induced cycle, namely $(4,14)_4$; and the second does not contain any induced cycle.

\ms\nt Suppose $l=2$, i.e., $s=13$. We get $m=104$, $y = 130$. The edge labels of the paths are
\[\begin{array}{lll}
1, 104, 26, 79, 51, 54, 76, 29, 101, 4; & 2, 103, 27, 78, 52, 53, 77, 28, 102, 3; & 5, 100, 30, 75, 55, 50, 80, 25;\\
6, 99, 31, 74, 56, 49, 81, 24; & 7, 98, 32, 73, 57, 48, 82, 23; & 8, 97, 33, 72, 58, 47, 83, 22;\\
9, 96, 34, 71, 59, 46, 84, 21; & 10, 95, 35, 70, 60, 45, 85, 20; & 11, 94, 36, 69, 61, 44, 86, 19;\\
12, 93, 37, 68, 62, 53, 87, 18; & 13, 92, 38, 67, 63, 42, 88, 17; & 14, 91, 39, 66, 64, 41, 89, 16;\\
15, 90, 40, 65.\end{array}\]
We can have $D_u = [1,8]\cup [22,25]$, $D_v = [10,14]\cup [16,19]$ and $D_w = \{9,15,20,21,65\}$.  The corresponding graph contains 12 induced cycles and one induced $(v,w)$-path, $(10,20)_8$.\rsq
\end{example}

\ms\nt Suppose $l\ge 3$.  Let \begin{align*}
D_u & = [1,l]\cup[9l+5,10l+4]\cup\{2l-1,2l,2l+4, (5l+3)(2l+1)\}\\
D_v &= [l+1,2l-2]\cup[2l+5, 3l+1]\cup[7l+4,9l+4]\cup\{2l+2,2l+3,10l+5\}\\
D_w & = \{2l+1\} \cup [3l+2,7l+3].\end{align*} It is routine to check that $\sg(D_z) = y$ for $z\in\{u,v,w\}$. Note that $|D_u| = 2l+4$, $|D_v| = 4l-1$ and $|D_w| = 4l+3$.
  %\\ $[l+1,2l-1]\cup [2l+3,9l+4]\cup \{10l+5\}$, {\cyan $[2l+3,8l+2] \cup [8l+3,8l+5] \cup [l+1,2l-1]\cup [8l+6,9l+4] \cup \{10l+5\}$} and $D_w = \{2l,2l+1,2l+2\}\cup [9l+5,10l+4]$

\ms\nt{\bf Case~C. }
Now, $s=4l+3$, $l\ge 1$. We get $y=(2l+1)(8l+5)$,   $D=[1, 8l+4]\cup\{\g_1, \g_2\}$ and $D'=[1, 8l+4]$, where $\g_2> \g_1>8l+4$ and $\g_1+\g_2=y$. There are $l$ induced paths of length $4l+2$ and $3l+2$ induced paths of length $4l$. We first choose an induced path and split it into two appropriate paths $Q_1,Q_2$ with end-edge labels $\g_1$ and $\g_2$, respectively. We then partition $D$ to obtain $D_u, D_v, D_w$ to get  a graph in $\widetilde{\mathcal B}_3$.

\ms \nt We first show that there is a solution for $s=7, 15$.

\begin{example}
\ms\nt Consider $s=7$, i.e., $l=1$. We have $y=39$. The only induced path of length 6 must have edge labels $1, 26, 13, 14, 25, 2$. We split it into $1, 26$ and $13, 14, 25, 2$. There are also five induced paths of length 4. Their consecutive edge labels are as follows.
\[\begin{array}{lllllll}
1, 26; & 2, 25, 14, 13; & 3, 24, 15, 12; & 4, 23, 16, 11; & 5, 22, 17, 10; & 6, 21, 18, 9; & 7, 20, 19, 8. \end{array}\]
Thus, we can have $D_u=\{1,6,8,11,13\}$, $D_v=\{2,4,5,7,9,12\}$ and $D_w=\{3,10,26\}$ to get a connected graph  without induced cycles.

\ms\nt Consider $s=15$, i.e., $l=3$. We have $y=203$  and $D'=[1, 28]$. There are three induced paths of length $14$ and eleven induced paths of length $12$.  We present them by using their end-edge labels notation:
\[(1,6),\ (2,5),\ (3,4),\ (6+j, 29-j), \mbox{ where } 1\le j\le 11.\]
Note that, the first three paths are of length $14$.
%\[\begin{array}{l}
%1, 174, 29, 146, 57, 118, 85, 90, 113, 62, 141, 34, 169, 6; \\
%2, 173, 30, 145, 58, 117, 86, 89, 114, 61, 142, 33, 170, 5; \\
%3, 172, 31, 144, 59, 116, 87, 88, 115, 60, 143, 32, 171, 4; \\
%7, 168, 35, 140, 63, 112, 91, 84, 119, 56, 147, 28; \\
%8, 167, 36, 139, 64, 111, 92, 83, 120, 55, 148, 27; \\
%9, 166, 37, 138, 65, 110, 93, 82, 121, 54, 149, 26; \\
%10, 165, 38, 137, 66, 109, 94, 81, 122, 53, 150, 25; \\
%11, 164, 39, 136, 67, 108, 95, 80, 123, 52, 151, 24; \\
%12, 163, 40, 135, 68, 107, 96, 79, 124, 51, 152, 23; \\
%13, 162. 41, 134, 69, 106, 97, 78, 125, 50, 153, 22; \\
%14, 161, 42, 133, 70, 105, 98, 77, 126, 49, 154, 21; \\
%15, 160, 43, 132, 71, 104, 99, 76, 127, 48, 155, 20; \\
%16, 159, 44, 131, 72, 103, 100, 75, 128, 47, 156, 19; \\
%17, 158, 45, 130, 73, 102, 101, 74, 129, 46, 157, 18.
%\end{array}\]
We can split the  path $(14, 21)$ whose consecutive edge labels are $14, 161, 42, 133, 70, 105, 98, 77, 126, 49, 154, 21$ into paths $(14,133): 14, 161, 42, 133$ and $(70, 21): 70, 105, 98, 77, 126, 49, 154, 21$. A solution is given by $D_u=[1,7]\cup[15,20]\cup\{70\}$, $D_v=\{13,14,21,22, 133\}$ and $D_w = [8,12]\cup [23, 28]$. \rsq
\end{example}

\ms\nt We now consider $s=4l+3$, $l\ne 1,3$.  Note that $y= (8l+5)(2l+1)$ and $D'=[1, 8l+4]$. Observe that every induced path of length $4l+2$ must have end-edge labels sum $2l+1$, and every induced path of length $4l$ must have end-edge labels sum $5(2l+1)$.

\ms\nt  We may choose $k$ such that $0\le k\le l$ and $8l+5-k\equiv 0\pmod{5}$, i.e., $8l+5 = 5t+k$ for some $t$. Clearly, if $l\ge 4$, then $t\ge l+2\ge 6$; if $l=2$, then $k=1$ and hence $t=4$.

\ms\nt Now, take any $t$ induced paths of length $4l$ and $k$ induced path(s) of length $4l+2$. Their end-edge labels sum is $(5t+k)(2l+1)=(8l+5)(2l+1)=y$. Merging their end vertices, we get $C((4l)^{[t]}, (4l+2)^{[k]})=G_1$ that has a vertex of degree $2(t+k)$ with induced vertex label $y$.
Thus, the remaining induced paths also have their end-edge labels sum is $y$.  Merging their end vertices, we get $C((4l)^{[3l+2-t]},(4l+2)^{[(l-k)]})=G_2$ that has a vertex of degree $2(4l+2-t-k)$ with induced vertex label $y$. Note that all these induced paths lead to part (a) solutions of our bridge graphs problem in~\cite{Lau+Shiu+Nal+Zhang+Prem}, and $G_1+G_2$ is a solution if we extend to non-bridge graphs too.

\ms\nt {\bf Splitting Approach~C: } For $g\ge 2$, consider $g$ end-edges of $G_1$ (or $G_2$)  with total labels sum $\sg$  such that $\sg$ satisfies the condition of Lemma~\ref{lem-split}. By Lemma~\ref{lem-split}, there is an induced path which can be split into $Q_1$ and $Q_2$. Let $\sg$ be the label of an edge $e$.  Now, swap the edge $e$ and the $g$ end-edges, say $e_1$ to $e_g$, so that these $g$ end-edges are now adjacent to the edge with label $y-\sg$ to get a vertex of degree $g+1$, whereas $e$ is now adjacent to all other $2(t+k) - g$ (or $2(4l+2-t-k)-g$) end-edges of $G_1$ (or $G_2$) to form a vertex of degree $2(t+k) - g+1$ (or $2(4l+2-t-k)-g+1$). Therefore, we have obtained a (possibly disconnected graph) solution with three vertices of degree $2(t+k) - g+1$, $g+1$, $2(4l+2-t-k)$, (or $2(4l+2-t-k)-g+1$, $g+1$, $2(4l+2-t-k)$), respectively.

 \ms\nt Since $D'=[1,8l+4], D^* = \{(2l+1)(8l+1)-\epsilon \mid \epsilon \in D'\} = [(2l+1)(8l-3), 2l(8l+5)]$, such integer $\sg$ exists because we may choose $g=2$ and $\sg=1+(8l+4)=8l+5$. Clearly, $\sg \not\in D'\cup D^*$ which satisfies the condition of Lemma~\ref{lem-split}.

\ms\nt{Case~D-1. } In this case, $s=3l$, $l\ge 2$, $m=4(3l-1)(l-1)$, $x=(2l-1)(6l-5)$, $y=2(3l-1)(2l-1)$. There are $l-1$ induced paths of length $4l-2$, $2l-1$ induced paths of length $4l-4$ and one induced path of length $2l-2$. Namely, they are
\[(i, 2l-1-i), 1\le i\le l-1;\quad (2l-2+j, 6l-2-j), 1\le j\le 2l-1;\quad (4l-2, (3l-1)(2l-1)).\]
Clearly, their end-edge labels sums are $2l-1$, $4(2l-1)$ and $(3l+1)(2l-1)$, respectively.

\begin{example} Consider $l=2$ with $y=30$. We begin with  one induced path of length 6, three induced paths of length 4 and one induced path of length 2, namely \[(1,2), (3,9), (4,8), (5,7), (6,15).\]
We can first partition the end-edge labels into $\{1,2,3,4,5,6,9\}$ and $\{7,8,15\}$. Since $4+6=10$ is the label in the path $1,20,10,11,19,2$, we then have $D_1 = \{1,2,3,5,9,10\}, D_2 = \{4,6,20\}, D_3 =\{7,8,15\}$ that gives a connected graph solution. \rsq
%1,20,10,11,19,2; & 3,18,12,9; & 4,17,13,8; & 5,16,14,7; & 6,15.
\end{example}

\ms\nt Consider $l\ge 3$. Now, $D'=[1, 6l-3]\cup\{(3l-1)(2l-1)\}$ and $D^*=[2(2l-1)(3l-4), 4(3l-1)(l-1)]\cup\{(2l-1)(3l-4)\}$. Similar to Case~C, we may choose $k$ such that $0\le k\le l-1$ and $6l-2-k\equiv 0\pmod{4}$, i.e., $6l-2=4t+k$ for some $t$.  Note that $t\ge 4$. Now, take any $t$ induced paths of length $4l-4$ and $k$ induced path(s) of length $4l-2$. Their end-edge labels sum is $(4t+k)(2l-1) = (6l-2)(2l-1) = y$. Merging the end-vertices, we get $C((4l-4)^{[t]}, (4l-2)^{[k]}) = G_1$ that has a vertex of degree $2(t+k)$ with induced vertex label $y$. Thus, the remaining induced paths also have their end-edge labels sum is $y$. Merging their end-vertices, we get $C((4l-4)^{[2l-1-t]}, (4l-2)^{[l-1-k]}, 2l-2) = G_2$ that has a vertex of degree $2(3l-1-t-k)$.  We may apply the Splitting Approach~C by choosing suitable end-edge labels with sum $\sg\notin D'\cup D^*$ to obtain a suitable solution $G$ for each possible $l$.

\begin{example} Take $l=3$, i.e., $s=9$.  Now $x=65$, $y=80$, $D'=[1,15]\cup\{40\}$ and $D^*=[50, 64]\cup\{25\}$. We have induced paths \[(1,4), (2,3) \mbox{ of length 10};\quad (5,15), (6,14), (7,13), (8,12), (9,11) \mbox{ of length 8};\quad (10, 40) \mbox{ of length 4}.\]
We may take $k =0$, $t=4$ and $G_1$ (and $G_2$) is obtained by merging the end-vertices incident to edges with labels in $[5,8]\cup[12,15]$ to get a vertex of degree 8 (and in $[1,4]\cup [9,11]\cup \{40\}$ to get a vertex of degree 8). Now, choose $\sg\in [16, 49]\setminus\{25,40\}$.
We may  choose three edge labels $5, 6, 7$. So $\sg = 5+6+7 = 18$ which is an edge label of the induced path $(2,3) : 2, 63, 17, 48, 32, 33, 47, 18, 62, 3$. Thus, $e$ with label $18$ is adjacent to edges with labels $47$ and $62$, respectively. We split the path $(2,3)$ to get $Q_1 = (2, 18)_8$ and $Q_2 = (62,3)_2$.  We can now have $D_u=[1,4]\cup \{9,10,11,40\}$, $D_v = \{5,6,7,62\}$ and $D_w = \{8,18\}\cup[12,15]$ that gives a connected graph solution. \rsq
\end{example}

\ms\nt{\bf Case~D-2. } Now, $s=3l+1$, $l\ge 2$. $y = 2l(6l-1)$, the edge label set is $[1, 6l-1]\cup \{l(6l-1)\}$. There are $3l-1$ induced paths of length $4l-2$, namely $(i,6l-1-i)$, $1\le i\le 3l-1$; and one induced path of length $2l-2$, namely, $(6l-1, l(6l-1))$.

\ms\nt We now show that there exists at least one solution for each $l\ge 2$. Observe that for  each induced path of length $4l-2$, the sum of end-edge labels is $6l-1$. Take any $2l-1$ of these paths and the end-vertex adjacent to label $6l-1$ of the path of length $2l-2$, merge these end-vertices to get a vertex of degree $4l-1$ with induced vertex label $(2l-1)(6l-1) + (6l-1) = y$.  Take the remaining $l$ paths of length $4l-2$ and the end-vertex adjacent to label $y/2$ of the path of length $2l-2$, merge these end-vertices to get a vertex of degree $2l+1$ with induced vertex label $l(6l-1) +(6l^2-l) = (12l^2-2l) = y$. This graph is obtained from  $C((4l-2)^{[2l-1]})$ and $C((4l-2)^{[l]})$ and connect the vertices of degree $4l-2$ and $2l$ by the path $(6l-1, l(6l-1)$. Applying the idea as in Case~C to obtain a suitable $Q_1$ and a $Q_2$ will give a required solution which may be a disconnected graph.

\begin{example}
When $s=7$, i.e., $l=2$. We get $y=44$. There are 5 induced paths of length 6 and one induced path of length 2, namely
\[(1,10),\ (2, 9),\ (3,8),\ (4,7),\ (5,6),\ (11,22).\]

\nt We can first partition the end-edge labels into $\{1,2,3,8,9,10,11\}$ and $\{4,5,6,7,22\}$.  Choose $\sg=8+9=17$ which is an edge label of the induced path $(5,6)_6 : 5, 28, 16, 17, 27, 6$.  We can have $D_u=\{1,2,3,10,11,17\}$, $D_v=\{4,5,6,7,22\}$, $D_w = \{8,9,27\}$ that gives a connected graph solution.

\ms\nt We may also begin with the partition $\{1,10,11,22\}$ and $\{2,3,4,5,6,7,8,9\}$. Since $10+11=21$ is the label in the path $(1,10)$ with consecutive edge labels 1, 32, 12, 21, 23, 10. We can split the path $(1,10)$ into $Q_1$ with edge labels $1,32,12,21$ and $Q_2$ with edge labels $23,10$ to have the partition $\{1,10,11,22\}$, $\{2,3,4,5,9,21\}$, $\{6,7,8,23\}$. Since $11$ and $22$ are in the same partite set and $2+9=11$, we can swap $2,9$ and $11$ to get the partition $D_u=\{1,2,9,10,22\}$, $D_v=\{3,4,5,11,21\}$, $D_w = \{6,7,8,23\}$.

\ms\nt A connected graph solution is obtained if we have the partition $D_u =\{1, 21, 22\}$, $D_v = \{6,7,8,23\}$, $D_w = \{2,3,4,5,9,10,11\}$. Another solution is given by $D_u = \{1,2,9,11,21\}$, $D_v = \{3,4,5,10,22\}$, $D_w = \{6,7,8,23\}$. \rsq\end{example}

\ms\nt{\bf Case~E. } In this case $s=2l+1$, $l\ge 2$,  $y=(2l-1)(4l-1)$ and $D'=[1, 4l-2]$.

\begin{example}\label{ex-E}
\nt For $s=5$, i.e., $l=2$, we get $y=21$, $x=15$, $m=14$. We begin with an induced path of length $6$ and two of length $4$ that have consecutive edge labels as follows.
\[\begin{array}{lll}
1,14,7,8,13,2; & 3,12,9,6; & 4,11,10,5. \end{array}\]
We split $3,12,9,6$ and $4,11,10,5$ into $3,12$; $9,6$; $4,11$; $10,15$ to get $D_u = \{1,9,11\}$, $D_v = \{2,3,6,10\}$ and $D_w = \{4,5,12\}$ to get a 2-connected graph solution. We can also split $1,14,7,8,13,2$ into $1,14$; $7,8$; $13,2$ to get $D_u = \{1,7,13\}$, $D_v=\{2,5,14\}$ and $D_w=\{3,4,6,8\}$ to get a 1-connected graph solution.

\ms\nt For $s=7$, i.e., $l=3$, we get $y=55$, $x=45$, $m=44$. We begin with two induced paths of length $10$ and three of length $8$ that have consecutive edge labels as follows.
\[\begin{array}{lll}
1,44,11,34,21,24,31,14,41,4; & 2,43,12,33,22,23,32,13,42,3; & \\ 5,40,15,30,25,20,35,10; & 6,39,16,29,26,19,36,9; & 7,38,17,28,27,18,37,8. \end{array}\]
We split first path into $1,44,11,34$ and $21,24,31,14,41,4$, and then split the second path into $2,43,12,33,22,23$, and $32,13,42,3$ to get a total of 7 paths. A 1-connected graph solution is given by $D_u=\{1,4,8,9,10,23\}$, $D_v=\{2,21,32\}$ and $D_w=\{3,5,6,7,34\}$.

\ms\nt We may also split $1,44,11,34,21,24,31,14,41,4$ into $1,44,11,34,21,24$ and $31,14,41,4$, and then split $5,40,15,30,25,20,35,10$ into $5,40$ and $15,30,25,20,35,10$ to get another solution given by $D_u=\{1,5,6,9,10,24\}$, $D_v=\{2,3,4,15,31\}$ and $D_w=\{7,8,40\}$. \rsq
\end{example}

\ms\nt  We now show that there exists a solution for each $l\ge 4$.  We have\\ $D^*=[(2l-1)(4l-5), (2l-1)(4l-3)-1]$.
Thus, there are
\begin{enumerate}[1)]
\item $l-1$ induced paths of length $4l-2$, $(i, 2l-1-i)$, for $1\le i\le l-1$, with end-edge label sum $2l-1$;
\item $l$ induced paths of length $4l-4$, $(2l-2+j, 4l-1-j)$, for $1\le j\le l$, with  end-edge label sum $3(2l-1)$.
\end{enumerate}

\nt Observe that for $z\ge 1$, any $2z$ induced paths of length $4l-4$ has end-edge labels sum  $2z(6l-3)=3z(4l-2)$. Moreover, an induced path of length $4l-2$ with $\a_1=i$ has $\b_{3z} =8l^2-10l+3-i-(3z-1)(4l-2) = 8l^2-6l+1-i-3z(4l-2)$ so that $\a_{3z+1} = i+3z(4l-2)$.

\ms\nt{\bf Splitting Approach E-1:} For $z,z'\ge 1$, $z+z'\le \lfloor\frac{l}{2}\rfloor$ and $1\le i\ne i' \le l-1$,  we can split the induced path $(i, 2l-1-i)$ into paths of length $6z$ and $4l-2-6z$, namely $(i,8l^2-6l+1-i-3z(4l-2))$ and $(2l-1-i, i+3z(4l-2))$; and split the induced path $(i', 2l-1-i'))$ into paths $6z'$ and $4l-2-6z'$, namely $(i', 8l^2-6l+1-i'-3z'(4l-2))$ and $(2l-1-i', i'+3z'(4l-2))$.

\ms\nt Now we may take $i=1$ and $i'=2$. We let $D_u = \{1, 8l^2-6l-3z(4l-2)\}\cup[2l-1, 2l-2+2z]\cup[4l-1-2z, 4l-2]$, $D_v = \{2, 8l^2-6l-1-3z'(4l-2)\} \cup [2l-1+2z, 2l-2+2z+2z']\cup [4l-1-2z-2z', 4l-2-2z]$ and $D_w = \{2l-2, 1+3z(4l-2), 2l-3, 2+3z'(4l-2)\} \cup [3,2l-4] \cup [2l-1+2z+2z', 4l-2-2z-2z']$.  It is routine to verify that sum of entries in each set is $y$. Thus, we have a 3-component graph solution.

\ms\nt\ms\nt{\bf Splitting Approach E-2:} Let $z\ge 1$, $z'\ge 0$, $z+z'\le (l-3)/2$. By a similar idea, a connected graph solution for each $l\ge 5$ can be obtained if  $D_u = \{1, 8l^2-6l-3z(4l-2)\}\cup[2l-1, 2l-2+2z]\cup[4l-1-2z, 4l-2]$ as above, $D_v=\{1+3z(4l-2), 18l-9+3z'(4l-2)\}\cup [2,2l-2]\cup [ 2l-1+2z, 3l-5-2z']\cup [3l+2+2z', 4l-2-2z]$ and $D_w = \{8l^2-24l+10-3z'(4l-2)\}\cup [3l-4-2z',3l+1+2z']$.  Note that if $z+z'=(l-3)/2$, then $D_v=\{1+3z(4l-2), 18l-9+3z'(4l-2)\}\cup [2, 2l-2]$.  Here, we split an induced path of length $4l-2$ (respectively $4l-4$) into paths  $(1, 8l^2-6l-3z(4l-2))_{6z}$ and $(2l-2, 1+3z(4l-2))_{4l-2-6z}$ (respectively, $(2l-1, 8l^2-24l+10+3z'(4l-2))_{8+6z'}$ and $(4l-2, 18l-9-3z'(4l-2))_{4l-12-6z'}$.  %{\cyan of length $6z$ and $4l-2-6z$ with end-edge labels $1, 8l^2-6l-3z(4l-2)$ and $2l-2, 1+3z(4l-2)$ (respectively, of length $8+6z'$ and $4l-12-6z'$ with end-edge labels $2l-1, 8l^2-24l+10+3z'(4l-2)$ and $4l-2, 18l-9-3z'(4l-2)$). Note that this approach can give a total of $\frac{k(k+1)}{2}$ $(k=\lfloor\frac{n-3}{2}\rfloor)$ distinct connected graph solutions.}

\begin{example} Take $s=15$, i.e., $l=7$. We get $y=351$, $x = 325$, $y-x = 26$.  We begin with $6$ induced paths of length $26$, and $7$ induced paths of length $24$ that have consecutive edge labels as follows.
{\small\[\begin{array}{l}
1, 324, 27, 298, 53, 272, 79, 246, 105, 220, 131, 194, 157, 168, 183, 142, 209, 116, 235, 90, 261, 64, 287, 38, 313, 12; \\
2, 323, 28, 297, 54, 271, 80, 245, 106, 219, 132, 193, 158, 167, 184, 141, 210, 115, 236, 89, 262, 63, 288, 37, 314, 11;\\
3, 322, 29, 296, 55, 270, 81, 244, 107, 218, 133, 192, 159, 166, 185, 140, 211, 114, 237, 88, 263, 62, 289, 36, 315, 10;\\
4, 321, 30, 295, 56, 269, 82, 243, 108, 217, 134, 191, 160, 165, 186, 139, 212, 113, 238, 87, 264, 61, 290, 35, 316, 9; \\
5, 320, 31, 294, 57, 268, 83, 242, 109, 216, 135, 190, 161, 164, 187, 138, 213, 112, 239, 86, 265, 60, 291, 34, 317, 8; \\
6, 319, 32, 293, 58, 267, 84, 241, 110, 215, 136, 189, 162, 163, 188, 137, 214, 111, 240, 85, 266, 59, 292, 33, 318, 7; \\
13, 312, 39, 286, 65, 260, 91, 234, 117, 208, 143, 182, 169, 156, 195, 130, 221, 104, 247, 78, 273, 52, 299, 26; \\
14, 311, 40, 285, 66, 259, 92, 233, 118, 207, 144, 181, 170, 155, 196, 129, 222, 103, 248, 77, 274, 51, 300, 25; \\
15, 310, 41, 284, 67, 258, 93, 232, 119, 206, 145, 180, 171, 154, 197, 128, 223, 102, 249, 76, 275, 50, 301, 24; \\
16, 309, 42, 283, 68, 257, 94, 231, 120, 205, 146, 179, 172, 153, 198, 127, 224, 101, 250, 75, 276, 49, 302, 23; \\
17, 308, 43, 282, 69, 256, 95, 230, 121, 204, 147, 178, 173, 152, 199, 126, 225, 100, 251, 74, 277, 48, 303, 22; \\
18, 307, 44, 281, 70, 255, 96, 229, 122, 203, 148, 177, 174, 151, 200, 125, 226, 99, 252, 73, 278, 47, 304, 21; \\
19, 306, 45, 280, 71, 254, 97, 228, 123, 202, 149, 176, 175, 150, 201, 124, 227, 98, 253, 72, 279, 46, 305, 20.
\end{array}\]}

\ms\nt Consider Splitting Approach E-1 that has $2\le z+ z' \le 3$. If we take $z=z'=1$, we split the first two induced paths into
\[\begin{array}{l}
1, 324, 27, 298, 53, 272; \\
79, 246, 105, 220, 131, 194, 157, 168, 183, 142, 209, 116, 235, 90, 261, 64, 287, 38, 313, 12; \\
2, 323, 28, 297, 54, 271; \\
80, 245, 106, 219, 132, 193, 158, 167, 184, 141, 210, 115, 236, 89, 262, 63, 288, 37, 314, 11 \end{array}\]
Thus, $D_u = \{1, 272\}\cup \{13, 14, 25, 26\}$, $D_v = \{2, 271\}\cup \{15,16,23,24\}$ and $D_w=\{11,12,79,80\}\cup[3,10]\cup[17,22]$.

\ms\nt Consider Splitting Approach E-2 that has $1\le z+z'\le 2$. If we take $z=2$, $z'=0$, we split the first and the 7th induced paths into
\[\begin{array}{l}
1, 324, 27, 298, 53, 272, 79, 246, 105, 220, 131, 194; \\
157, 168, 183, 142, 209, 116, 235, 90, 261, 64, 287, 38, 313, 12; \\
13, 312, 39, 286, 65, 260, 91, 234; \\
117, 208, 143, 182, 169, 156, 195, 130, 221, 104, 247, 78, 273, 52, 299, 26. \end{array}\]
Thus, $D_u =  \{1, 194\}\cup \{13, 14, 15, 16, 23, 24, 25, 26\}$, $D_v = \{117,157\}\cup [2,12]$ and $D_w = \{234\}\cup [17,22]$.

\ms\nt If we take $z=z'=1$, we split the first and the 7th induced paths into
\[\begin{array}{l}
1, 324, 27, 298, 53, 272; \\
79, 246, 105, 220, 131, 194, 157, 168, 183, 142, 209, 116, 235, 90, 261, 64, 287, 38, 313, 12; \\
13, 312, 39, 286, 65, 260, 91, 234, 117, 208, 143, 182, 169, 156; \\
195, 130, 221, 104, 247, 78, 273, 52, 299, 26. \end{array}\]
Thus, $D_u =  \{1, 272\}\cup \{13, 14, 25, 26\}$, $D_v = \{79, 195\}\cup [2,12]$ and $D_w = \{156\}\cup [15,24]$.

\ms\nt If we take $z=1, z'=0$, we split the first and the 7th induced paths into
\[\begin{array}{l}
1, 324, 27, 298, 53, 272; \\
79, 246, 105, 220, 131, 194, 157, 168, 183, 142, 209, 116, 235, 90, 261, 64, 287, 38, 313, 12; \\
13, 312, 39, 286, 65, 260, 91, 234; \\
117, 208, 143, 182, 169, 156, 195, 130, 221, 104, 247, 78, 273, 52, 299, 26.\end{array}\]
Thus, $D_u =  \{1, 272\}\cup \{13, 14, 25, 26\}$, $D_v = \{79, 117\}\cup [2,12]\cup \{15,16,23,24\}$ and $D_w = \{234\}\cup [17,22]$. \rsq
\end{example}

\nt Thus, we have some solutions for each possible $l$ of Case~E-1. For the solution of Case~E-2, we may consider $D_u=[1, 4l-2]$ to get $\sg(D_u)=y$. Pick one induced path of length at least 6. Choose two vertices with induced label $y$. Rename them by $v$ and $w$, respectively. We can then obtain a graph in $\widetilde{\mathcal B} _3$.

\begin{example}
Let us consider $s=5,7$ again. When $l=2$, there is only one induced path of length 6, namely $(1,2)$. So we split it into 3 parts: $1, 14$; $7,8$; $13, 2$. Thus, we can have $D_u=\{1,2, 3,4,5,6,\}$, $D_v=\{14, 7\}$, $D_w=\{8, 13\}$.

\ms\nt When $l=3$, we may choose the induced path $(1,4)$ and split it into 3 parts: $1,44,11,34$; $21,24,31,14$; $41,4$. Thus, we can have $D=[1,10]$, $D_v=\{34, 21\}$ and $D_w=\{14, 41\}$. \rsq
\end{example}

\ms\nt{\bf Case~F. } We now give a simple way to construct a solution for each $s\ge 5$.

%\begin{example}

\ms\nt For $s=5$, we have $y=30$, $x=24$. So, we begin with three induced paths of length 10 and an induced path of length 4 that have consecutive edge labels as follows.
\[\begin{array}{lll}
1, 24, 6, 19, 11, 14, 16, 9, 21, 4; & 2, 23, 7, 18, 12, 13, 17, 8, 22, 3; & 5, 20, 10, 15.
\end{array}\]
We then split the first two paths into two paths of length 4 and another two of length 6 to get the followings.
\[\begin{array}{lllll}
1, 24, 6, 19, 11, 14; & 16, 9, 21, 4; & 2, 23, 7, 18;  & 12, 13, 17, 8, 22, 3; & 5, 20, 10, 15.
\end{array}\]
Thus, we can have $D_u=\{1, 14, 15\}$, $D_v=\{2, 12,16\}$ and $D_w=\{3, 4, 5, 18\}$. This is a connected graph solution. %\rsq
%\end{example}

\ms\nt Consider $s\ge 6$. First, we split the induced path with $\a_1=1$ at $\b_{s-2}$ into a path of length $2s-4$ with end-edge labels $1$ and $(2s-5)^1-1 - (s-3)(2s-5) = 2s^2-9s+9 = (s-3)(2s-3) = y/2-1$, and a path of length $2s-6$ with end-edge labels $1+(s-2)(2s-5) = 2s^2-9s+11 = y/2+1$ and $2s-6$. Next, we split the induced path with $\a_1 = 3$ at $\b_{s-2}$ into a path of length $2s-4$ with end-edge labels $3$ and $y/2-3$, and a path of length $2s-6$ with end-edge labels $y/2+3$ and $2s-8$. Now, we can have $D_u = \{1, y/2-1, y/2\}$, $D_v=\{2, y/2-3, y/2+1\}$ and $D_v = [3,2s-5]\cup \{y/2+3\}$. It is routine to check that each of $D_u, D_v, D_w$ has total elements sum is $y$. This is also a connected graph solution.

\begin{example} Take $s=6$, we have $y=56$, $x=49$. After splitting, we can have the 6 required induced paths with consecutive edge labels as follows.
\[\begin{array}{lll}
1, 48, 8, 41, 15, 34, 22, 27; & 29, 20, 36, 13, 43, 6; & 2, 47, 9, 40, 16, 33, 23, 26, 30, 19, 37, 12, 44, 5; \\
3, 46, 10, 39, 17, 32, 24, 25; & 31, 18, 38, 11, 45, 4;  & 7, 42, 14, 35, 21, 28.
\end{array}\]
Now, $D_u = \{1, 27, 28\}$, $D_v = \{2, 25, 29\}$ and $D_w=[3,7]\cup\{31\}$. We can extend the same idea to get more non-isomorphic graph solutions . For example, we can split the paths with an end-edge label 1 and 7 as follows:
\[\begin{array}{lll}
1, 48, 8, 41, 15, 34, 22, 27, 29, 20; & 36, 13, 43, 6; & 2, 47, 9, 40, 16, 33, 23, 26, 30, 19, 37, 12, 44, 5; \\
3, 46, 10, 39, 17, 32, 24, 25, 31, 18, 38, 11, 45, 4;  & 7, 42; &  14, 35, 21, 28.
\end{array}\]
We can have $D_u=\{1,7,20,28\}$, $D_v = \{2,3,4,5,42\}$ and $D_w = \{6, 14, 36\}$. \rsq
\end{example}

\nt Thus, we have some solutions for each possible $s$ under Case~F-1. By a similar approach as Case~E, we may construct some solutions for each possible $s$ under Case~F-2.
\end{proof}

\section{Conclusion and Future Directions}

In this paper, we successfully obtained necessary and sufficient conditions for a large family of bipartite graph with local antimagic chromatic number 2. To our best knowledge, this is the first successful approach (in the field of graph labelings) in determining the edge labeling of a family of infinitely many graphs by algebraic analysis in which we can extend further to obtain all the possible solutions in $\widetilde{\mathcal B}_d$ for every $d\ge 4$, once the value of $d$ and the size of the graph $m$, are fixed. Moreover, the labeling is unique for all edges, i.e., once an edge label is determined, its neighboring and subsequent edge labels are determined. Thus, we have the following problems.

\begin{problem} For each $d\ge 4$, find necessary and sufficient conditions for each $G\in\widetilde{\mathcal B}_d$ without case-by-case analysis as in Theorem~\ref{thm-iff}. \end{problem}

\nt Since characterization of $G\in \widetilde{\mathcal B}_2$ is not yet completed, we also have the following problem.

\begin{problem} For each $d\ge 2$, give complete characterization of $G\in \widetilde{\mathcal B}_d$. \end{problem}

\ms \nt This gives partial solutions to~\cite[Problem 4.1]{Lau+Shiu+Nal+Zhang+Prem}. Our proofs for sufficiency also shows the existence of infinitely many disconnected such graphs (see Cases A-2 and B-1). This is also a natural extension to the results obtained by Ba\v{c}a et. al ~\cite{Baca+A+W} that gives the bounds for copies of graphs.  We shall in another paper gives all the 2- and 3-component solutions in $\widetilde{\mathcal B}_d, d = 2,3$, respectively. Specifically, our current results and subsequent paper will be a natural extension to the existence of infinitely many (possibly regular) disconnected  bipartite and tripartite graphs without pendant vertices that have local antimagic chromatic number 3 (see~\cite{Chan+Lau+Shiu, Chan+Shiu+Lau, Lau+Shiu+ArsComb, Lau+Shiu+CN, Lau+Shiu+UM, Lau+Shiu+Pre+Nal, Lau+Shiu+Nal+Pre, XLZ-IJMSI}).

\ms\nt In~\cite{Chan+Shiu+Lau}, the authors also showed that $\chi_{la}(2C_3+C_k) = 4$ for $k\ge 4$ and $\chi_{la}(3C_3) = 5$. We end this paper with the following problem that arises naturally.

\begin{problem} Show the existence of infinitely many (regular) disconnected graphs $G$ without pendant vertices such that $\chi(G) = k$ and $\chi_{la}(G) {\Lau > k}$ for each $k\ge 2$.  \end{problem}

\end{document}